\pdfoutput=1
\documentclass{article}
\usepackage[utf8]{inputenc}
\usepackage{amsmath, amsthm, amssymb,color,dsfont, mathtools, yfonts, amsfonts}
\usepackage[all]{xy}
\usepackage{calligra, mathrsfs}
\usepackage[nodisplayskipstretch]{setspace}
\usepackage{thmtools}
\usepackage{thm-restate}
\usepackage{caption}
\usepackage{subcaption}
\usepackage{scrextend}
\addtokomafont{labelinglabel}{\sffamily}
\usepackage{geometry}
\usepackage{marginnote}
\usepackage{hyperref}
\usepackage[noabbrev, capitalise]{cleveref}
\usepackage[titletoc]{appendix}

\usepackage[style=alphabetic, doi=false,isbn=false,url=false, maxcitenames = 99, maxbibnames = 99, natbib]{biblatex}
\usepackage{graphicx, floatrow, wrapfig}
\usepackage{xstring}

\input{svg-macros}

\allowdisplaybreaks

\newtheorem{introthm}{Theorem}
\newtheorem*{introcor}{Corollary}
\newtheorem{introq}{Question}
\newtheorem{thm}{Theorem}[section]
\newtheorem{lem}[thm]{Lemma}
\newtheorem{prop}[thm]{Proposition}
\newtheorem{cor}[thm]{Corollary}
\theoremstyle{definition}

\theoremstyle{definition}
\newtheorem{defn}[thm]{Definition}
\theoremstyle{remark}
\newtheorem{rmk}[thm]{Remark}
\theoremstyle{remark}

\newcommand{\ti}{\overline{t}_1}
\newcommand{\tii}{\overline{t}_2}
\newcommand{\tiii}{\overline{q t}_3}
\newcommand{\tiv}{\overline{t}_4}

\newcommand{\Sk}{\operatorname{Sk}_s}
\newcommand{\End}{\operatorname{End}}
\newcommand{\Ad}{\operatorname{Ad}}
\newcommand{\lp}{\operatorname{loop}}
\newcommand{\qint}[1]{\left[#1\right]_s}
\newcommand{\rspan}{\operatorname{span}_R}
\newcommand{\SL}{\operatorname{SL}_2}

\newcommand{\hs}{\hat \sigma}
\newcommand{\hy}{\hat y}

\newcommand*\centermathcell[1]{\omit\hfil$\displaystyle#1$\hfil\ignorespaces}

\definecolor{twcol}{rgb}{0.6, 0, 0.2}

\title{On the genus two skein algebra}
\author{Juliet Cooke and Peter Samuelson}

\usepackage[en-GB]{datetime2}
\DTMlangsetup{showdayofmonth=false}
\date{\today}

\begin{document}
\maketitle
\abstract{We study the skein algebra of the genus 2 surface and its action on the skein module of the genus 2 handlebody. We compute this action explicitly, and we describe how the module decomposes over certain subalgebras in terms of polynomial representations of double affine Hecke algebras. Finally, we show that this algebra is isomorphic to the $t=q$ specialisation of the  genus two spherical double affine Hecke algebra recently defined by Arthamonov and Shakirov. }

\section{Introduction}

The Kauffman bracket skein algebra of a surface $\Sigma$ is spanned by framed links in the thickened surface $\Sigma \times [0,1]$, with the following local\footnote{By \emph{local relation} we mean the following. The first picture represents 3 links which are identical outside of an embedded ball in $\Sigma \times [0,1]$ and inside the ball are as pictured. Similarly, the second says a trivially framed unknot in an embedded ball can be removed at the expense of a scalar.}
 relations imposed: 
\begin{align*}
\skeindiagram{leftnoarrow} &= s\; \skeindiagram{horizontal} + s^{-1} \; \skeindiagram{vertical} \\ 
\skeindiagram{circle}      &= -s^{2} -s^{-2}
\end{align*}
\noindent 

Multiplication in the algebra is given by stacking links in the $[0,1]$ direction. Typically, this algebra is noncommutative; however, its $s=\pm 1$ specialisation is commutative since in this case the right hand side of the first skein relation is symmetric with respect to switching the crossing. \citeauthor{Bul97,PS00}~\cite{Bul97,PS00} showed that at $s=-1$ the skein algebra $\Sk(\Sigma)$ is isomorphic to the ring of functions on the $\SL$ character variety of $\Sigma$. \citet{BFK99} strengthened this statement by showing that the skein algebra is a quantization of the $\SL$ character variety of $\Sigma$ with respect to Atiyah--Bott--Goldman Poisson bracket. 

The \emph{skein module} of a 3-manifold $M$ is defined in the same way as the skein algebra (using links in $M$ instead of in $\Sigma \times [0,1]$), and it is a module over the algebra associated to the boundary $\partial M$. The action is given by `pushing links from a neighbourhood of the boundary into $M$'. At $s=-1$, the 3-manifold $M$ determines a Lagrangian subvariety of the character variety, which consists of the representations $\pi_1(\partial M) \to \SL(\mathbb C)$ that extend to representations of $\pi_1(M)$. The fact that the skein module of $M$ is a module over $\Sk(\partial M)$ is an illustration of the general principle that `the quantization of a coisotropic subvariety is a module'.

From now on we write $\Sigma_{g,n}$ for the genus $g$ surface with $n$ punctures. For the torus $\Sigma_{1,0}$,
the skein algebra and its action on the skein module of the solid torus were described explicitly by \citet{FG00}.  
This description led to interesting connections to Double Affine Hecke Algebras (DAHAs); for example, it follows from their results that the skein algebra of the torus is isomorphic to the $t=q$ specialisation of the $A_1$ spherical DAHA, the $\SL(\mathbb Z)$ actions on both algebras agree, and the polynomial representation of the DAHA is isomorphic (again at $t=q$) to the skein module of the solid torus. The main goal of the present paper is to find relationships between various versions of DAHAs and the skein algebra and module of genus 2 surface and handlebody.

Since DAHAs tend to be defined by explicit formulas, our first task is to find concrete descriptions of the skein module $\Sk(\mathcal H_2)$ of the handlebody, which we do as follows.
The skein algebra $\Sk(\Sigma_{2,0})$ of the closed genus two surface is generated by the five curves $A_1,A_2,A_3, B_{12},$ and $B_{23}$ depicted in \cref{fig:introfigure}. 
There are two natural bases to use for the skein module $\Sk(\mathcal H_2)$ of the handlebody, the theta basis $n(i,j,k)$ and the dumbbell basis $m(i,j,k)$. 
We compute the action of generators of $\Sk(\Sigma_{2,0})$ on $\Sk(\mathcal H_2)$ in both of these bases in the following theorem (see \cref{naction} and \cref{maction}).
\begin{introthm}\label{thm:intro1}
The operators $A_1, A_2, A_3$ act diagonally in the theta basis $n(i,j,k)$ of $\Sk(\mathcal H_2)$, and matrix coefficients for $B_{12}$, $B_{13}$, $B_{23}$ 
with respect to this basis are in equation \eqref{eq:nact}. 
The dumbbell basis $m(i,j,k)$ diagonalises the operators $A_1$ and $A_3$, and matrix coefficients for the other operators are in \cref{maction}.
\end{introthm}
There is a third basis of $\Sk(\mathcal H_2)$ which comes from the fact that the genus 2 handlebody is diffeomorphic to a thickening of a 2-punctured disc. 
This means $\Sk(\mathcal H_2)$ is actually an algebra, and it turns out to be the polynomial algebra in the elements $B_{12}$, $B_{13}$, and $B_{23}$.
We provide some comments about the resulting monomial basis in \cref{sec:background} following \cref{lemma:poly}.

\subsection{Double affine Hecke algebras of rank 1}
Double affine Hecke algebras were introduced by \citet{Che95} for his proof of Macdonald's conjectures, and have since been related to a wide variety of areas (see, for example \cite{Che05} and references therein). There are two versions that will be relevant for our purposes: the 2-parameter spherical\footnote{The spherical DAHA is a subalgebra  which is analogous to a subalgebra of invariants of a group action.} DAHA $SH_{q,t}$ of type $A_1$, and the 5-parameter spherical DAHA $S\mathscr H_{q,t_1,t_2,t_3,t_4}$ of type $(C^\vee_1, C_1)$. 
These algebras each have a \emph{polynomial representation}, which is an analogue of Verma modules in Lie theory. 

Terwilliger has given presentations of both spherical DAHAs, and combining these with results of \citet{Bullock&Przytycki} one can construct algebra maps $\Sk(\Sigma_{1,1}) \to SH_{q,t}$ and $\Sk(\Sigma_{0,4}) \to S\mathscr H_{q,\{t_i\}}$. 
This implies the skein algebras of these punctured surfaces act on the polynomial representations of the corresponding spherical DAHAs. 
These skein algebras also act on $\Sk(\mathcal H_2)$ since they 
map\footnote{Note that there are two copies of $\Sigma_{1,1}$ embedded in $\Sigma_{2,0}$, which means we have two commuting actions of the skein algebra $\Sk(\Sigma_{1,1})$ on $\Sk(\mathcal{H}_2)$.} 
to the skein algebra of the genus 2 surface via the surface maps in \cref{fig:introfigure}. 

\begin{figure}[!h]
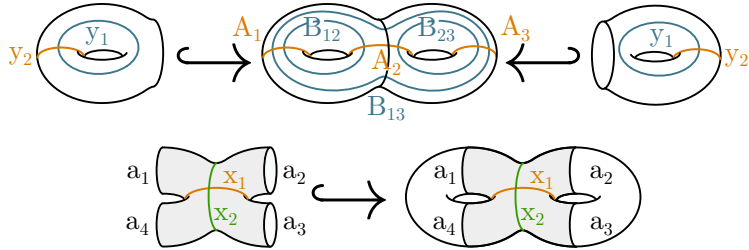

 \begin{center}
    \diagramhh{figures}{introemb}{0pt}{0pt}{0.3} 
\end{center}
    \caption{Loops and surface embeddings}\label{fig:introfigure}
\end{figure}

\noindent 
Using the structure constants from \cref{thm:intro1}, we compute how the skein module decomposes over the subalgebras corresponding to these subsurfaces in the following (see \cref{cor:irreducible}, \cref{thm:2a1decomp}, and \cref{cor:ccdecomp}):
\begin{introthm} As a module over $\Sk(\Sigma_{2,0})$, the skein module $\Sk(\mathcal H_2)$ is irreducible.
\begin{enumerate}
\item As a module over $\Sk(\Sigma_{1,1})\otimes \Sk(\Sigma_{1,1})$, the skein module decomposes as a direct sum
\[
\Sk(\mathcal H_2) = \bigoplus_{j \geq 0} P_j \otimes P_j
\]
where $P_j$ is the specialisation of the $A_1$ polynomial representation at $t=-q^{-j-2}$. \\
\item As a module over $\Sk(\Sigma_{0,4})$, the skein module decomposes as a direct sum 
\[
\Sk(\mathcal H_2) = \bigoplus_{i,k \geq 0} P'_{i,k}
\]
where $P'_{i,k}$ is the (unique) finite dimensional quotient of the spherical $(C^\vee_1, C_1)$ polynomial representation with parameter specialisations given in equation \eqref{eq:params}.
\end{enumerate}
\end{introthm}

We aren't aware of an explanation or construction of this module using classical DAHA theory, so we briefly summarise some of its unusual features here. Terwilliger's universal Askey-Wilson algebra $\Delta_q$ maps to the skein algebra of the punctured torus and to the skein alegbra of the 4-punctured sphere.
Therefore, there are 3 copies of $\Delta_q$ acting simultaneously on the space $R[B_{12},B_{23}, B_{13}]$ of polynomials in 3 variables. 
These algebras have nontrivial intersections which corresponds to the intersections of the respective subsurfaces. For example, the `Casimir' element for both the left and right copies of $\Delta_q$ correspond to the curve $A_2$ above; however, this curve corresponds to the Askey-Wilson operator in the middle copy of $\Delta_q$. 
Conversely, central elements in the middle copy of $\Delta_q$ act as the loops labelled $A_1$ and $A_3$, which correspond to the Macdonald operators in the left and right copies of $\Delta_q$. 
We also point out that these three subalgebras generate the skein algebra of the genus 2 surface (see \cref{cor:gen}), and that over this algebra, the skein module of the handlebody is irreducible (see \cref{cor:irreducible}). Finally, we note that the left and right copies of $\Delta_q$ contain the  multiplication operators $B_{12}$ and $B_{23}$ respectively. However, the $B_{13}$ multiplication operator isn't in any of the 3 subalgebras in question; instead, it has to be written as a fairly complicated expression involving generators of all three copies of $\Delta_q$ (see \cref{cor:dehn}).

Finally, we recall that a \emph{Leonard pair} is a finite dimensional vector space $V$ with two diagonalizable endomorphisms $A,B$ such that $A$ has a tridiagonal\footnote{A tridiagonal matrix only has nonzero entries which are on or adjacent to the diagonal.} matrix with respect to an eigenbasis of $B$, and vice-versa (see \cite{Ter01}). There is an extensive literature on Leonard pairs, and they have arisen in representation theory, combinatorics, orthogonal polynomials, and more (see, e.g. \cite{Ter03} and references therein). We show Leonard pairs also appear in skein theory; in particular, in \cref{cor:leonard} we show that each $P'_{i,k}$ is a Leonard pair with respect to the operators $x_1$ and $x_2$. It seems likely that these particular Leonard pairs have appeared before, e.g. in \cite{NT17}. However, using the topological point of view in the present paper, it is evident that $\oplus_{i,k \geq 0} P'_{i,k}$ is a module for $\Delta_q\otimes \Delta_q$, and it is not clear if this observation has appeared in the literature.

\subsection{The genus 2 double affine Hecke algebra}
If we combine the Terwilliger presentation of the $A_1$ spherical DAHA $SH_{q,t}$ with the Frohman-Gelca description of the skein algebra of the torus, it follows almost immediately that the skein algebra $\Sk(\Sigma_{1,0})$ is isomorphic to the $t=q=s$ specialisation $SH_{s,s}$ of the $A_1$ spherical DAHA. Furthermore, the $\SL(\mathbb Z)$ actions on both algebras agree, and the polynomial representation of the spherical DAHA in this specialisation is isomorphic to the skein module of the solid torus. These results are surprising since the objects on the DAHA side are defined in terms of explicit formulas, while the skein-theoretic definitions are purely topological.
Our second main motivation for the present paper was to generalise these results to genus 2.

In genus 2, this comparison was not possible until the work of \citet{arthamonov&shakirov2017}, who recently gave a very interesting proposal for a definition of the
\emph{genus 2 spherical DAHA}.
They define their algebra in terms of its action on a space with basis $\Psi_{i,j,k}$ where $(i,j,k)$ ranges over the set of  admissible triples\footnote{See \cref{def:adm}; this definition also comes up in the bases we use for the skein module.}. 
Their algebra depends on two parameters, $q$ and $t$, and we use our skein-theoretic results to prove the following (see \cref{thm:AScompare}): 
\begin{introthm}\label{thm:introspec}
The $q=t$ specialisation of the Arthamonov-Shakirov algebra is isomorphic to the image of the skein algebra $\Sk(\Sigma_{2,0})$ in the endomorphism ring of $\Sk(\mathcal H_2)$.
\end{introthm}

Thang Le \cite{Le21} has established the faithfulness of the action of the skein algebra of a (closed) surface on the skein module of the corresponding handlebody. 
Le's result combined with \cref{thm:introspec} immediately imply the following corollary (see \cref{cor:asisoskein}).

\begin{introcor}
Using the faithfulness result in \cite{Le21}, the skein algebra $\Sk(\Sigma_{2,0})$  is isomorphic to the $t=q$ specialisation of the Arthamonov--Shakirov genus 2 DAHA.
\end{introcor}
Arthamonov and Shakirov raised a number of questions about their algebra; in particular, they ask \cite[Pg. 17]{arthamonov&shakirov2017} whether their algebra is a flat deformation of the skein algebra. Our corollary above proves that the $t=q$ specialisation is isomorphic to the skein algebra;
however, to the best of our knowledge, it is still unknown whether their algebra is a \emph{flat} deformation of the skein algebra.

\subsection{Future directions}
The results described above lead or contribute to some interesting questions in both representation theory and knot theory. 
On the representation theory side, in \cite{terwilliger2013} Terwilliger has defined a \emph{universal Askey-Wilson algebra} $\Delta_q$, and showed that this algebra surjects onto the spherical DAHA $S \mathscr H_{q,\{t_i\}}$ (and hence, onto the $A_1$ spherical DAHA also). 
In fact, these surjections factor through the skein algebras (see \cref{rmk:universal}), so we have maps $\Delta_q \to \Sk(\Sigma_{0,4})$ and $\Delta_q \to \Sk(\Sigma_{1,1})$. 
This implies we have three maps from $\Delta_q$ to the skein algebra $\Sk(\Sigma_{2,0})$ of the genus 2 surface.
\begin{introq}
Do the three maps $\Delta_q \to \Sk(\Sigma_{2,0})$ deform to maps from $\Delta_q$ to the Arthamonov-Shakirov algebra? How does their module decompose over the subalgebras given by the images of these maps?
\end{introq}

In knot theory, there have been quite a number of papers conjecturing and/or proving a relationship between DAHAs and knot invariants for torus knots 
(beginning\footnote{In this paragraph, citations grouped together within square brackets have been sorted by chronological order with respect to arXiv posting, since some articles experienced significant publication delays.} 
with \cite{AS15, CerednikJones}), iterated torus knots (beginning with \cite{Sam19, CD16}), and iterated torus links \cite{CD17}. 
One of the recent successes in this direction was the work of \citeauthor{Mel17, HM19}~\cite{Mel17,HM19}
which proved  that the Khovanov--Roszansky homology for positive torus knots/links can be computed using the elliptic Hall algebra. 
Roughly, the (Euler characteristic of) Khovanov--Roszansky homology and the elliptic Hall algebra are the `$\mathfrak{gl}_\infty$' analogues of the Kauffman bracket knot polynomial and $A_1$ spherical DAHA, respectively, both of which correspond to $\mathfrak{sl}_2$. 

From our point of view, the  heuristic for these conjectures and results is that the torus knot $T_{m,n}$ is embedded in the torus, and can therefore be viewed as an operator on the skein module of the solid torus. Using the $\SL(\mathbb Z)$ action on the spherical DAHA, an analogous operator $\hat T_{m,n}$ can be defined as an element of the DAHA. The real surprise is that the operator $\hat T_{m,n}$ and its action in the polynomial representation can compute the \emph{Poincare polynomial} of the knot homology of the torus knot (instead of the Euler characteristic, which is computed by the skein algebra action). This leads to the following question, which was asked slightly differently in the last sentence of \cite{arthamonov&shakirov2017}:
\begin{introq}[\cite{arthamonov&shakirov2017}] Can the Poincare polynomial of Khovanov homology of genus 2 knots (such as the figure eight knot) be computed using the Arthamonov-Shakirov algebra?
\end{introq}
The present paper provides some evidence that this question has a positive answer: our \cref{thm:introspec} above shows that the answer is yes at the level of Euler characteristics. In other words, we have the following corollary, which is stated precisely in \cref{cor:Jones}.
\begin{introcor}
If $\alpha$ is a simple closed curve on $\Sigma_{2,0}$, then the Jones polynomial\footnote{The standard embedding of $\Sigma_{2,0}$ into $S^3$ allows us to interpret a curve on $\Sigma_{2,0}$ as a knot in $S^3$.} of $\alpha$ can be computed using the Arthamonov-Shakirov algebra.
\end{introcor}

Finally, we mention that Hikami has also given a proposal for a genus 2 DAHA by gluing together the $A_1$ and $(C^\vee C_1)$ spherical DAHAs, and has conjectured \cite[Conj. 5.5]{Hik19} that his algebra can also be used to compute (coloured) Jones polynomials of knots embedded on the genus 2 surface. However, the relation between his construction and the Arthamonov-Shakirov construction is not clear, so the results of the present paper do not seem to be immediately applicable to this conjecture.

An outline of the paper is as follows. In \cref{sec:background}, we recall background material on skein theory and double affine Hecke algebras, and perform some initial computations. In \cref{sec:calculations}, we state theorems giving matrix coefficients of actions of certain loops on the skein module. In \cref{sec:Handdaha} we use DAHAs to describe how the skein module decomposes over certain subsurfaces, and we briefly discuss Leonard pairs. In \cref{ASlink} we show how our skein-theoretic calculations are related to the genus 2 DAHA defined by Arthamanov and Shakirov. Finally, the appendices contain diagrammatic proofs of the matrix coefficient computations.

\noindent \textbf{Acknowledgements:} We would like to thank the referee for their careful reading and helpful comments and remarks. We would like to thank Paul Terwilliger for his interest and his guidance through the literature regarding Leonard pairs and Thang Le for proving the faithfulness result we use and for helpful conversations. We would also like to thank Semeon Arthamonov, Matt Durham, David Jordan, Gregor Masbaum,  Shamil Shakirov for many helpful discussions over the past several years, and Thomas Wright for careful proofreading. Both authors were partially supported by the ERC grant 637618, the second author was partially supported by a Simons Foundation Collaboration Grant, and the first author was funded by a EPSRC studentship and the F.R.S.-FNRS.

\section{Background}\label{sec:background}
\subsection{Kauffman Bracket Skein Modules}
\label{back:wenzl}
Kauffman bracket skein modules are based on the Kauffman bracket:
\begin{defn}
Let \(L\) be a link without contractible components (but including the empty link). The \emph{Kauffman bracket polynomial} \(\langle L \rangle\) in the variable \(s\) is defined by the following local \emph{skein relations}:
\begin{align}
  \skeindiagram{leftnoarrow} &= s\; \skeindiagram{horizontal} + s^{-1} \; \skeindiagram{vertical}, \label{skeinrel_cross} \\ 
  \skeindiagram{circle} &= -s^{2} -s^{-2}. \label{skeinrel_loop}
\end{align}
\end{defn}
\noindent (These diagrams represent three links which are identical outside of the dotted circles and are as pictured inside the dotted circles.) It is an invariant of framed links and it can be `renormalised' to give the Jones polynomial. The Kauffman bracket can also be used to define an invariant of \(3\)-manifolds:
\begin{defn}
Let \(M\) be a 3-manifold, \(R\) be a commutative ring with identity and \(s\) be an invertible element of \(R\). The \emph{Kauffman bracket Skein module} \(\Sk(M; R)\) is the \(R\)-module of all formal linear combinations of links, modulo the Kauffman bracket skein relations pictured above.
\end{defn}

\begin{rmk}
For the remainder of the paper we will use the coefficient ring \(R := \mathbb Q(s)\).
\end{rmk}

We now define the Jones-Wenzl idempotents which one can use to construct a diagrammatic calculus for skein modules. 
This diagrammatic calculus is heavily used in \cref{appendix:cals} and \cref{appendix:finalloop} to calculate the loop actions used throughout this paper. 

\begin{defn}
The \emph{Jones-Wenzl idempotent} is depicted using a box and is defined recursively as follows:
\begin{align*}
    \diagramhhink{wenzl}{wenzl1}{18pt}{0pt}{0.6} &= \diagramhhink{wenzl}{wenzl2}{18pt}{0pt}{0.6} - \frac{\qint{n}}{\qint{n+1}} \diagramhhink{wenzl}{wenzl3}{18pt}{0pt}{0.6} \\
    \diagramhhink{wenzl}{wenzl4}{18pt}{0pt}{0.6} &= \diagramhhink{wenzl}{wenzl5}{18pt}{0pt}{0.6}
\end{align*}
where a strand labelled by an integer $n>0$ depicts $n$ parallel strands and by an integer $n \leq 0$ depicts $0$ strands. 
\end{defn}
In the definition above we have used the constants below.
\begin{defn}
Let \(n\) be an integer. The \emph{quantum integer} \(\qint{n}\) is
\[\qint{n} = \frac{s^{2n} - s^{-2n}}{s^2 - s^{-2}}\]
\end{defn}
\begin{rmk}
Note that \(\qint{0} =0\), \(\qint{1} =1\) and \(\qint{-n}= -\qint{n}\).
\end{rmk}

\begin{defn}
For any non-negative integer \(n\) the quantum factorial \(\qint{n}!\) is defined as
\[\qint{n}! = \qint{n} \qint{n-1} \dots \qint{1},\]
so in particular \(\qint{0}! = 1\).
\end{defn}

Jones-Wenzl idempotents are also used to define trivalent vertices which are used to describe generating sets of the skein module of a handlebody.
\begin{defn}\label{def:adm}
A triple \((a,b,c)\) is \emph{admissible} if \(a,b,c \geq 0\), \(a+b+c\) is even and \(|a-b| \leq c \leq a+b\).
We denote by \(\Ad \subseteq \mathbb{N}_0^3\) the set of all admissible triples.
\end{defn}
\begin{defn}[\cite{masbaum1994}]
Given an admissible triple \((a,b,c)\) one can define the \emph{3-valent vertex}:
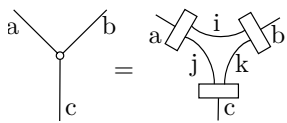
\[\diagramhh{figures}{2valent}{0pt}{0pt}{0.6}\]
where \(i = \frac{(a+b-c)}{2}\), \(j= \frac{(a+c-b)}{2}\), \(k= \frac{(b+c-a)}{2}\) are integers as \((a,b,c)\) is admissible.
\end{defn}

\subsection{Skein Modules of Handlebodies}
\label{sec:skeinmodulebasis}
Using the trivalent vertices described in the previous section one can describe a generating set for the skein module of a handlebody.

\begin{defn}
Let \(\mathcal{H}_g\) denote the solid \(g\)-handlebody.
\end{defn}

\begin{thm}[\cite{lickorish1993, zhong2004}]
\label{thm:HandlebodyBasis}
A generating set of \(\Sk(\mathcal{H}_g)\) is given by 
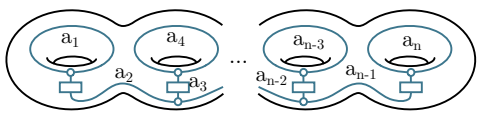
\[
\diagramhh{figures}{nbasis}{0pt}{0pt}{0.3}
\]
where \(a_1, \dots, a_n\)  are non-negative integers which form admissible triples around every 3-valent node.
\end{thm}

Other generating sets  of \(\Sk(\mathcal{H}_g)\) can be obtained by modifying the diagram in \cref{thm:HandlebodyBasis} at two  adjacent  nodes using  quantum \(6j\)-symbols \cite{kauffmanBook1994} as  the  change of  base  matrix. This change of basis is useful  for checking calculations, and it will also been needed in \cref{ASlink} as a different basis is used by Arthamonov and Shakirov.

\begin{thm}[Change of Basis \cite{kauffmanBook1994}]
\label{thm:ChangeofBasis}
If \((r,t,j)\) and \((s,u,j)\) are admissible triples then
\[
\smalldiagram{properties}{basis2} = \sum_a \left\{ \begin{smallmatrix}t & r & a \\ s & u & j \end{smallmatrix} \right\} \smalldiagram{properties}{basis1}
\]
where the sum is over all \(a\) such that \((r, s, a)\) and \((t, u, a)\) are admissible.
\end{thm}

Summarising the discussion so far, we have two bases of skein module of the genus 2 handlebody which are depicted in \cref{fig:bases} and are related via the change of basis formula given in \cref{thm:ChangeofBasis}. The fact that these are bases (and not just generating sets) has been well-known to experts for a long time; we have, however, not found a precise reference so we sketch a proof in \cref{lemma:isbasis} below.
\begin{figure}[!h]
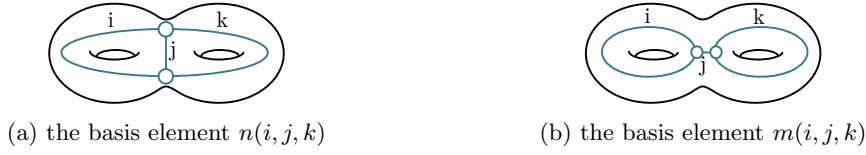

\centering
\begin{subfigure}{.4\textwidth}
  \centering
  \diagramhh{figures}{2basis}{0pt}{0pt}{0.3}
  \caption{the basis element \(n(i, j, k)\)} 
\end{subfigure}
\hspace{5mm}
\begin{subfigure}{.4\textwidth}
  \centering
  \diagramhh{figures}{2basis2}{0pt}{0pt}{0.3}
  \caption{the basis element \(m(i, j, k)\)}
\end{subfigure}
\caption{Bases for the skein module of the handlebody}\label{fig:bases}
\end{figure}


In fact, the skein of the solid $2$-torus is an algebra, since the handlebody is diffeomorphic to an interval crossed with a twice-punctured disc. We recall its algebra structure in the following.
\begin{lem}[{\cite[Prop. 1(6)]{Bullock&Przytycki}}]\label{lemma:poly}
As an algebra, $\Sk (\mathcal H_2)$ is isomorphic to $R[B_{12}, B_{23}, B_{13}]$ (using the notation of \cref{fig:cycles}), a polynomial algebra in 3 variables.
\end{lem}
In the present notation, we have
\begin{align*}
B_{12} &= n(1,1,0) = m(1,0,0),\\
B_{23} &= n(0,1,1) = m(0,0,1),\\
B_{13} &= n(1,0,1) = m(1,2,1) + \frac 1 {\qint{2}} m(1, 0, 1)
\end{align*}
(In terms of the dumbbell basis $m(i,j,k)$, the expression for the loop $B_{13}$  can be computed using the Jones-Wenzl recursion.) 
Since $\Sk(\mathcal H_2)$ is isomorphic to a polynomial algebra, this gives us a natural monomial basis. We can partially describe the change of basis matrix between the monomial basis and our bases $m(i,j,k)$ and $n(i,j,k)$ as follows. Let $S_n(x)$ be the Chebyshev polynomials, which are determined uniquely by the condition
\[
S_n(X+X^{-1}) = \frac{X^{n+1}-X^{-n-1}}{X-X^{-1}}
\]
We then have the following identities: 
\[
m(i,0,0) = n(i,0,0) = S_i(B_{12}), \quad \quad n(k,0,k) = S_k(B_{13}),\quad \quad m(0,0,k) = n(0,0,k) = S_k(B_{23})
\]
In general the change of basis matrix between monomials in $B_{\ell \ell'}$ and either $m(i,j,k)$ or $n(i,j,k)$ is determined uniquely by the formulas in \cref{naction}, which express the action of $B_{\ell \ell'}$ as a matrix in the $n(i,j,k)$ basis. However, the change of basis matrices are not so easy to write explicitly outside of the cases above.

As mentioned above, the following lemma is well-known to experts, but for the convenience of the reader we provide a sketch of a proof. 
\begin{lem}\label{lemma:isbasis}
The set $n(i,j,k)$ is a basis for $\Sk(\mathcal H_2)$.
\end{lem}
\begin{proof}
 
By \cref{lemma:poly}, the set $\{B_{12}^a B_{23}^b B_{13}^c\}$ (for $a,b,c \in \mathbb N$) is a basis for $\Sk(\mathcal H_2)$. Using the recursion defining the Jones-Wenzl idempotents, one can see that $B_{12}^a B_{23}^b B_{13}^c =n(a+c,a+b,b+c)$ plus lower order terms (with respect to the lexicographic order on triples of nonegative integers). 
This shows the matrix converting from the monomial basis to the theta basis is upper-triangular, with ones on the diagonal. (The fact that the diagonal entries are $1$ follows from the fact that the recursion defining the Jones-Wenzl idempotents has a 1 as the first coefficient.)
\end{proof}

\subsection{Skein Algebras of Simple Punctured Surfaces}
If \(\Sigma\) is a surface then the skein algebra \(\Sk(\Sigma \times [0,1]; R)\) forms an algebra with multiplication given by stacking the links on top of each other to obtain a link in \(\Sigma \times [0,2]\), then rescaling the second coordinate to obtain \(\Sigma \times [0,1]\) again.

\begin{defn}
Let \(\Sigma_{g,n}\) denote the surface with genus $g$ and $n$ punctures.
\end{defn}

Presentations are known for skein algebras of a small number of surfaces. We shall use the presentations for the four-punctured sphere $\Sigma_{0,4}$ and \(1\)-punctured torus $\Sigma_{1,1}$, which we recall below. These presentations all use the $q$-Lie bracket, defined as follows:
\[
[a,b]_q := qab - q^{-1} ba
\]

Let \(a_i\) denote the loops around the four punctures of \(\Sigma_{0,4}\), and let \(x_i\) denote the loops around punctures 1 and 2, 2 and 3, 1 and 3 respectively (see \cref{figure:FPuncSphereEdges}). If curve $x_i$ separates $a_i,a_j$ from $a_k, a_\ell$, let
$ p_i = a_i a_j + a_k a_\ell$. Explicitly,
\[
p_1 = a_1 a_2 + a_3 a_4,\quad \quad p_2 = a_2 a_3 + a_1 a_4,\quad \quad p_3 = a_1 a_3 + a_2 a_4
\]
We now recall the following theorem\footnote{We have corrected a sign error in the first relation which appears in the published version of the paper \cite{Bullock&Przytycki}.} of Bullock and Przytycki.

\begin{thm}[\cite{Bullock&Przytycki}]
\label{thm:presentationofsphereskein}
As an algebra over the polynomial ring $R[a_1,a_2,a_3,a_4]$, the Kauffman bracket skein algebra \(\Sk(\Sigma_{0,4})\) has a presentation with generators \(x_1,\,x_2,\,x_3\) and  relations
\begin{align*}
    \left[x_i, x_{i+1}\right]_{s^2} &= (s^4 - s^{-4})x_{i+2} + (s^2 - s^{-2})p_{i+2} \text{ (indices taken modulo 3)}; \\
    \Omega_K &= (s^2 +s^{-2})^2 - (p_1 p_2 p_3 p_4 + p_1^2 + p_2^2 + p_3^2 + p_4^2);
\end{align*}
where we have used the following `Casimir element':
\[\Omega_K := -s^2 x_1 x_2 x_3 + s^4 x_1^2 + s^{-4} x_2^2 + s^4 x_3^2 + s^2 p_1 x_1 + s^{-2}p_2 x_2 + s^2 p_3 x_3.\]
\end{thm}

\begin{figure}[h]
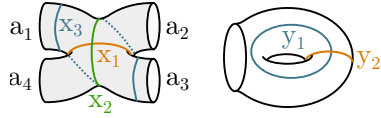

\centering
\diagramhh{figures}{skeinpres}{0pt}{0pt}{0.3}
\caption{The generating loops for $\Sk(\Sigma_{0,4})$ and $\Sk(\Sigma_{1,1})$}
\label{figure:FPuncSphereEdges}
\end{figure}
Let \(y_1\) and \(y_2\) denote the loops around the meridian and longitude respectively of a torus. These loops cross once and resolving this crossing gives \(y_1 y_2 = s y_3 + s^{-1} z\). 

\begin{thm}[\cite{Bullock&Przytycki}]
\label{thm:torusskeinpres}
As an algebra over $R$, the Kauffman bracket skein algebra \(\Sk(\Sigma_{1,1})\) has a presentation with generators \(y_1, y_2, y_3\) and relations
\[[y_i, y_{i+1}]_{s} = (s^2-s^{-2}) y_{i+2} \text{ (indices taken modulo 3)}.\]\end{thm}

The loop \(p\) around the puncture of \(\Sigma_{1,1}\) is obtained from the resolution of the crossing of \(z\) and \(y_3\) using the identity \(z y_3 = s^2 y_1^2 + s^{-2} y^2_2 -s^2 -s^{-2} + p\).

\subsection{Generation of skein algebras}
In this subsection we briefly discuss sets of curves which generate the skein algebra $\Sk(\Sigma_{2,0})$. In \cite{San18}, Santharoubane gave the following very useful criteria for showing a set of curves generates the skein algebra of a surface.

\begin{thm}[\cite{San18}]
Let $\{\gamma_j\}_{j \in I}$ be a finite set of non-separating simple closed curves such that the following conditions hold:
\begin{enumerate}
    \item For any $i,j \in I$, the curves $\gamma_i$ and $\gamma_j$ intersect at most once,
    \item The set of Dehn twists around the curves $\gamma_i$ generate the mapping class group of $\Sigma$.
\end{enumerate}
Then the curves $\gamma_i$ generate the skein algebra of $\Sigma$.
\end{thm}

Using this theorem we can prove:
\begin{cor}\label{cor:gen}
The skein algebra $\Sk(\Sigma_{2,0})$ is generated by each of the following:
\begin{enumerate}
    \item The set of curves $I := \{A_1, A_2, A_3, B_{12}, B_{23}\}$ of \cref{fig:cyclesall},
    \item The subalgebras $\Sk(\Sigma_{1,1}^L)$, $\Sk(\Sigma_{1,1}^R)$, and $\Sk(\Sigma_{0,4})$ of \cref{fig:embeddings}.
\end{enumerate}
\end{cor}
\begin{proof}
In \cite{Hum79}, Humphries constructed a finite generating set of the mapping class group for closed surfaces, and in the genus 2 case it is exactly the set of curves in the first claim. The second claim follows since each of the curves in $I$ is contained in one of the subalgebras mentioned.
\end{proof}

Santharoubane proves his generation result by using some theory that has been developed for mapping class groups. These groups are related to skein theory using \cref{lem:dtoskein}, which is a standard lemma relating Dehn twists to the resolutions of the crossing that are given by the skein relation. 
If $\alpha$ is a simple closed curve, let $D_\alpha$ be the right-handed Dehn twist along $\alpha$. In particular, if $\beta$ is a simple closed curve intersecting $\alpha$ once, then $D_\alpha(\beta)$ is the simple closed curve which can be described in words as `follow $\beta$, turn right at the intersection and follow $\alpha$, then turn right at the intersection and continue following $\beta$.' 

\begin{lem}\label{lem:dtoskein}
Let $\alpha$ and $\beta$ be two simple closed curves in $\Sigma$ that intersect once. We have the following relation in the skein algebra of $\Sigma$:
\[
D_\alpha^\epsilon(\beta) = \epsilon \frac {s^\epsilon \alpha \beta - s^{-\epsilon} \beta \alpha}{(s^2-s^{-2})}
\]
where $\epsilon \in \{\pm 1\}$.
\end{lem}

We now give an explicit computation of the loop $B_{13}$ in terms of the other loops in \cref{fig:cycles}. (We note that a  computation for $B_{13}$ also appeared in \cite[Fig. 7]{Hik19}, which used an identity in the mapping class group that is similar to \eqref{eq:mcg}.) 
\begin{cor}\label{cor:dehn}
We have the following identity in the skein algebra of $\Sigma_{2,0}$:
\[
B_{13} = -\delta^{-4}[A_3,[B_{23},[A_1,[A_2,B_{12}]_{s^{-1}}]_s]_{s^{-1}}]_{s^{-1}}
\]
where $\delta = s^2-s^{-2}$.
\end{cor}
\begin{proof}
Elementary computations with Dehn twists show the following identity:
\begin{equation}\label{eq:mcg}
B_{13} = (D_{A_3}^{-1}\circ D_{B_{23}}^{-1}\circ D_{A_1}\circ D_{A_2}^{-1})(B_{12})
\end{equation}
Then the claim follows from \cref{lem:dtoskein}
\end{proof}

\subsection{Spherical Double Affine Hecke Algebras}
Double Affine Hecke Algebras (DAHAs) were introduced by \citet{Che95}, who used them to prove Macdonald's constant term conjecture for Macdonald polynomials. These algebras have since found wider ranging applications particularly in representation theory \cite{Che05}. DAHAs were associated to different root systems, and we recall the type $A_1$ case here.

\begin{defn}
The \emph{\(A_1\) double affine Hecke algebra} \(H_{q,t}\)  is the algebra generated by \(X^{\pm 1}\), \(Y^{\pm 1}\) and \(T\) subject to the relations
\[T X T = X^{-1}, \quad T Y^{-1} T = Y, \quad XY = q^2 T^{-2} Y X, \quad (T - t)(T + t^{-1}) = 0\]
\end{defn}

\begin{rmk}
Our presentation here varies slightly from the standard presentation in \cite{Che05}, and in particular replaces his $t$ with $t^{-1}$. See \cite[Rmk. 2.19]{Berest&Samuelson} for the precise relation.
\end{rmk}

The element \(e = (T + t^{-1})/(t + t^{-1})\) is an idempotent of \(H_{q,t}\), and is used to define the spherical subalgebra \(SH_{q, t} := e H_{q, t} e\).

We shall also consider Sahi's \cite{Sah99} DAHA associated to the $(C^\vee_1, C_1)$ root system (see also \cite{NS04} for the rank 1 case that we use in this paper).
The double affine Hecke algebra \(\mathscr{H}_{q,\underbar{t}}\) of type $(C^\vee_1, C_1)$ is a 5-parameter universal deformation of the affine Weyl group \(\mathbb{C}[X^{\pm}, Y^{\pm}] \rtimes \mathbb{Z}_2\) with the deformation parameters \(q \in \mathbb{C}^*\) and \(\underbar{t} = (t_1, t_2, t_3, t_4) \in (\mathbb{C}^*)^4\) and it is a generalisation of Cherednik's double affine Hecke algebras of rank 1, since there is an isomorphism \(H_{q;t} = \mathscr{H}_{q, 1,1,t, 1}\).

\begin{defn}
The \emph{$(C^\vee_1, C_1)$ double affine Hecke algebra} is the algebra with generators  \(T_0, T_1, T_0^{\vee}, T_1^{\vee}\) and the following relations:
\begin{align*}
     (T_0-t_1)(T_0 + t^{-1}_1) &= 0 \\
     (T_1-t_3)(T_1 + t^{-1}_3) &= 0 \\
     (T_0^{\vee}-t_2)(T_0^{\vee} + t^{-1}_2) &= 0 \\
     (T_1^{\vee}-t_4)(T_1^{\vee} + t^{-1}_4) &= 0 \\
     T_1^{\vee} T_1 T_0 T_0^{\vee} &=q
\end{align*}
\end{defn}
The spherical subalgebra\footnote{The spherical subalgebra is not a unital subalgebra; instead,  the unit in the spherical subalgebra is the idempotent $e$.} $ S\mathscr{H}_{q,\underbar{t}} $ is defined in terms of an idempotent $e \in \mathscr{H}_{q,\underline t}$ as follows:
\begin{equation}\label{eq:edef}
e := (T_1 + t^{-1}_3)/(t_3 + t_3^{-1}),\quad \quad S\mathscr{H}_{q, \underbar{t}} := e \mathscr{H}_{q, \underbar{t}} e
\end{equation}

Terwilliger gave presentations of the spherical DAHAs which will be useful for us (for the conversion between our notation and Terwilliger's notation, see \cite{Sam19} and \cite{BS16}). Define the following elements in $S\mathscr{H}_{q,\underbar t}$:
\[x := (q^{-1}T_0 T_0^\vee + q (T_0 T_0^\vee)^{-1})e, \qquad
    y := (T_1 T_0 + T_0^{-1}T_1^{-1})e, \qquad
    z := (T_1 T_0^{\vee} + (T_1 T_0^{\vee})^{-1})e\]
In what follows it will be helpful to use the following notation:
\[
\ti := t_1 - t_1^{-1}, \quad \tiii = qt_3 - q^{-1}t_3^{-1},\quad \mathrm{etc.}
\]
\begin{thm}[\cite{terwilliger2013}] \label{thm:sDAHA pres}
The spherical double affine Hecke algebra \(S\mathscr{H}_{q, \underbar{t}}\) has a presentation with generators \(x,y,z\) defined above and relations
\begin{align*}
    [x,y]_q &= (q^2 - q^{-2})z - (q-q^{-1})\gamma \\
    [y,z]_q &= (q^2 - q^{-2})x - (q-q^{-1})\alpha \\
    [z,x]_q &= (q^2 - q^{-2})y - (q-q^{-1})\beta \\
    \Omega &= (\ti)^2 + (\tii)^2 + (\tiii)^2 + (\tiv)^2 - \ti \,\tii\, (\tiii)\, \tiv + (q + q^{-1})^2
\end{align*}
where \(\alpha := \ti\, \tii + \tiii \, \tiv\), \(\beta := \ti\, \tiv + \tiii\, \tii\), \(\gamma := \tii\, \tiv + \tiii\, \ti\), and the `Casimir' \[\Omega := -qxyz + q^2 x^2 + q^{-2} y^2 + q^2 z^2 - q \alpha x - q^{-1} \beta y - q \gamma z.\]
\end{thm}
Using Terwilliger's presentation of the spherical DAHA and the Bullock-Przytycki presentation of the skein algebra of the 4-punctured sphere (\cref{thm:presentationofsphereskein}), we obtain the following result.
\begin{prop}
\label{prop:Bereset&SamuelsonIso}
There is an algebra map \(\varphi: \Sk(\Sigma_{0,4}) \to S\mathscr{H}_{q,\underline{t}}\) given by
\begin{alignat*}{5}
 \varphi(x_1) &= x \qquad&\varphi(x_2) &= y \qquad &\varphi(x_3) &= z \qquad&\varphi(s) &= q^2,\\
 \varphi(a_1) &= i \ti &\varphi(a_2) &= i \tii &\varphi(a_3) &= i \tiv &\varphi(a_4) &= i (\tiii).
\end{alignat*}
where $i^2=-1$. Under this map,
\[
p_1 \mapsto - \alpha,\quad \quad p_2 \mapsto - \beta, \quad \quad p_3 \mapsto - \gamma
\]
\end{prop}
\begin{rmk}
To the best of our knowledge, \cref{prop:Bereset&SamuelsonIso} first appeared in notes of Terwilliger which have not been published. It was also stated in \cite{Berest&Samuelson}, but the notational conventions there are slightly different. In particular, the images of $a_3$ and $a_4$ are different here and in \cite{Berest&Samuelson}, but because of the underlying differences in notation, both statements are correct. This statement also was proved directly in \cite[Thm. 3.8]{Hik19}, without relying on Terwilliger's presentation of the spherical DAHA.
\end{rmk}

Using \cite[Rmk. 2.19]{Berest&Samuelson}, it isn't too hard to see that (in our current conventions) the $A_1$ DAHA is the specialisation of the $(C^\vee_1, C_1)$ DAHA at $\underbar t = (1,1,t,1)$. This means we can specialise the presentation of the spherical DAHA, and the generators specialise as follows:
\[
x = (X+X^{-1})e,\quad \quad y = (Y+Y^{-1})e,\quad \quad z = (qYX + q^{-1}X^{-1}Y^{-1})e
\]
These generators lead to the following presentation.
\begin{thm}[\cite{terwilliger2013}]
The spherical double affine Hecke algebra \(SH_{q, t}\) has a presentation with generators \(x,y,z\) and relations
\[[x, y]_q = (q^2 - q^{-2}) z, \quad  [z, x]_q = (q^2 - q^{-2})y, \quad [y, z]_q = (q^2 - q^{-2})x\]
\[q^2 x^2 + q^{-2} y^2 + q^2 z^2 - qxyz = \left( \frac{t}{q} - \frac{q}{t} \right)^2 + \left(q + \frac{1}{q} \right)^2\]
\end{thm}
Note that in the following statement we identify $s = q$ instead of $s = q^2$.
\begin{prop}[\cite{Sam19}]\label{prop:sktodahaA}
There is an algebra map \({\Sk}(\Sigma_{1,1}) \to SH_{q, t}\) uniquely determined by the following assignments:
\[ y_1 \mapsto x, \quad y_2 \mapsto y, \quad y_3 \mapsto z,\quad p \mapsto -q^2 t^{-2} - q^{-2} t^2, \quad s \mapsto q
\]
\end{prop}

\begin{rmk}\label{rmk:universal}
Terwilliger \cite{terwilliger2013} has introduced a \emph{universal Askey-Wilson algebra} $\Delta_q$, which maps to the spherical $(C^\vee_1, C_1)$ DAHA. 
The algebra $\Delta_q$ is generated by elements $A,B,C$, and the relations state that $A + [B,C]_q/(q^2-q^{-2})$ (along with its two cyclic permutations $A \mapsto B \mapsto C$) are central. He showed that this algebra maps to the spherical $(C^\vee_1, C_1)$ DAHA (where $A,B,C$ map to $x,y,z$), and it follows that it also maps to the spherical $A_1$ DAHA. Using the presentations of the skein algebras of $\Sigma_{1,1}$ and $\Sigma_{0,4}$, it is clear that these maps factor through the maps from the skein algebras in \cref{prop:Bereset&SamuelsonIso} and \cref{prop:sktodahaA}.
\end{rmk}

\subsection{Polynomial representations of DAHAs}\label{sec:polyrep}
The skein algebras of $\Sigma_{0,4}$ and $\Sigma_{1,1}$ act on the skein of the genus 2 handlebody using the maps of surfaces in  \cref{fig:embeddings}. We would like to use Propositions \ref{prop:Bereset&SamuelsonIso} and \ref{prop:sktodahaA} to identify these modules in terms of representations of DAHAs. In this section we recall the polynomial representation of $\mathscr{H}_{q,\underbar t}$ on $R[X^{\pm 1}]$ from \cite{NS04} and compute some structure constants of this action. 

First we define two auxiliary operators $\hs, \hy \in \mathrm{End}_R(R[X^{\pm 1}])$:
\begin{equation}
    \hs(f(X)) := f(X^{-1}),\quad \quad \hy(f(X)) := f(q^{-2}X)
\end{equation}
We write the actions of the operators $T_0$ and $T_1$ in terms of $\hs$, $\hy$, and the multiplication operator $X$:
\begin{align}
    T_0 &:= t_1 \hs \hy - \frac{q^2 \ti X^2 + q \tii X}{1-q^2 X^2} (1 - \hs \hy)\label{eq:t0op}\\
    T_1 &:= t_3 \hs + \frac{\overline{t}_3 + \tiv X}{1-X^2} (1-\hs)\notag
\end{align}
(Here we have used the notation $\bar{t}_i = t_i - t_i^{-1}$). 
We note that \emph{a priori} these operators act on rational functions $R(X)$, but in fact they preserve the subspace $R[X^{\pm 1}]$ of Laurent polynomials, since $(1-\hs \hy)f(X)$ is always divisible by $1-q^2 X^2$, and similar for $T_1$. Since $T_0$, $T_1$, and $X$ generate the DAHA, these definitions completely determine the action, and it is shown in \cite{NS04} that these operators satisfy the relations of the DAHA~$\mathscr{H}_{q,  \underline t}$.

The spherical subalgebra $S\mathscr{H}_{q,\{t_i\}}$ acts on $eR[X^{\pm 1}]$ and it isn't hard to see that the idempotent $e$ projects onto the subspace $R[X+X^{-1}] = R[x]$ of symmetric Laurent polynomials.

\begin{lem}\label{lem:bc1dahaact}
We have the following identities:
\begin{align}
    y\cdot 1 &= t_1 t_3 + t_1^{-1} t_3^{-1}\\
    yx\cdot 1 &= (q^2 t_1^{-1} t_3^{-1} + q^{-2} t_1 t_3) x + \left[ -\ti \tiv + \tiv (q^{-2} t_1 -q^2 t_1^{-1}) + (q^{-1}-q)\tii (t_3 + t_3^{-1}) \right]
\end{align}
\end{lem}
\begin{proof}
The first equation is immediate, and the second is straightforward but somewhat tedious. Using the quadratic relations for $T_1$ and $T_0$, we need to compute
\begin{equation}\label{eq:yx}
(T_1T_0 + (T_1 T_0)^{-1})\cdot (X+X^{-1}) = (T_1 T_0 + T_0 T_1 - \ti T_1 - \overline{t}_3 T_0 + \ti \overline{t}_3)\cdot (X+X^{-1})
\end{equation}
Using equations \eqref{eq:t0op}, we can compute
\begin{alignat*}{2}
    &T_0\cdot X &&= t_1 q^{-2} X^{-1}  + \ti X + \tii q^{-1}\\
    &T_0\cdot X^{-1} &&= q^2 t_1^{-1} X - q \tii\\
    &T_1 \cdot X &&= t_3^{-1} X^{-1} - \tiv\\
    &T_1 \cdot X^{-1} &&= t_3 X + \overline{t}_3 X^{-1} + \tiv
\end{alignat*}
The claimed identity is the result of substituting these four equations into \eqref{eq:yx} and simplifying.
\end{proof}

Since the $A_1$ DAHA is the $\underline t = (1,1,t,1)$ specialisation of the $(C^\vee_1, C_1)$ DAHA, we have the following corollary to \cref{lem:bc1dahaact}:

\begin{cor}\label{cor:a1act}
The $A_1$ spherical DAHA acts on $R[x]$, and under this action we have
\begin{align}
    y\cdot 1 &= t+t^{-1}\label{cor:a1dprep}\\
    yx\cdot 1 &= (q^2 t^{-1} + q^{-2}t)x \notag
\end{align}
\end{cor}
\section{Loop Actions}
\label{sec:calculations}

In this section we shall determine the actions of the loops depicted in 
\cref{fig:cycles} and \cref{fig:cycles2} on \(\Sk(\mathcal{H}_2)\), the skein module of the solid \(2\)-handlebody. 
These loop actions define the actions of the skein algebras $\Sk(\Sigma_{0,4})$ and $\Sk(\Sigma_{1,1})$ on $\Sk(\mathcal{H}_2)$ which we will use in \cref{sec:rep} to decompose $\Sk(\mathcal{H}_2)$. 
We shall also relate these loop actions to the operators generating the Arthamonov--Shakirov, genus $2$, spherical DAHA in \cref{ASlink}.
For the loops in \cref{fig:cycles}, we shall use the theta b{}asis for \(\Sk(\mathcal{H}_2)\), which is the set \(\{n(i,j,k)\}\) for all admissible triples \((i,j,k)\).
For the loops in \cref{fig:cycles2}, we use the dumbbell basis for \(\Sk(\mathcal{H}_2)\), which is the set \(\{m(i,j,k)\}\) for all \(i,j,k\) such that \((i,i,j)\) and \((k, k, j)\) are admissible\footnote{Note that the restrictions on $i,j,k$ are exactly those required for the edges entering each trivalent vertex to form an admissible triple, and thus for the trivalent vertex to be well defined.}.

\begin{figure}[!h]
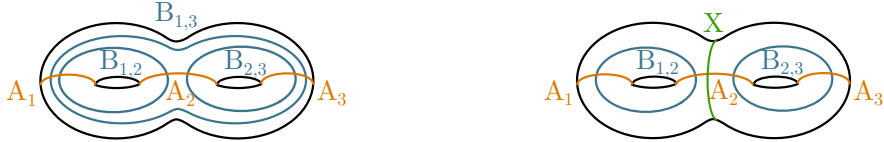

\centering
\begin{subfigure}{.4\textwidth}
  \centering
  \diagramhh{figures}{cycles}{0pt}{0pt}{0.3}
  \caption{A set of generators of $\Sk(\Sigma_{2,0})$ which we act on $\Sk(\mathcal{H}_2)$ with basis $n(i,j,k)$.}
  \label{fig:cycles}
\end{subfigure}
\hspace{5mm}
\begin{subfigure}{.4\textwidth}
  \centering
  \diagramhh{figures}{cycles2}{0pt}{0pt}{0.3}
  \caption{Another set of generators of $\Sk(\Sigma_{2,0})$ which we act on $\Sk(\mathcal{H}_2)$ with basis $m(i,j,k)$.}
  \label{fig:cycles2}
\end{subfigure}
\caption{Generating loops}\label{fig:cyclesall}
\end{figure}

\begin{defn}
Let $(i, j, k)$ be an admissible triple. We define the coefficients $D_{a,b}(i,j,k)$ for $a,b = \pm 1$ as follows
\begin{alignat*}{5}
D_{1,-1}(i,j,k) &= -\frac{\qint{\frac{j+k-i}{2}}^2}{ \qint{j} \qint{j+1}} 
&&D_{-1,-1}(i,j,k) &&= \frac{\qint{\frac{i+j+k+2}{2}}^2 \qint{\frac{i+j-k}{2}}^2}{ \qint{i} \qint{i+1} \qint{j} \qint{j+1}} \\
D_{-1,1}(i,j,k)\;&= -\frac{\qint{\frac{i+k-j}{2}}^2}{ \qint{i} \qint{i+1}}
&\qquad \quad& \centermathcell{D_{1,1}(i,j,k)} &&= 1
\end{alignat*}
where the coefficient is defined to be $0$ if the denominator is $0$.
\end{defn}

\begin{thm}
\label{naction}
Let $(i, j, k)$ be an admissible triple. The $A$-loops act on the theta basis of $\Sk(\mathcal{H}_2)$ by scalars
\begin{align*}
     A_1 \cdot n(i,j,k) &= (-s^{2i+2} -s^{-2i-2}) n(i,j,k)\\
    A_2 \cdot n(i,j,k) &= (-s^{2j+2} -s^{-2j-2}) n(i,j,k)\\
    A_3 \cdot n(i,j,k) &= (-s^{2k+2} -s^{-2k-2}) n(i,j,k)
\end{align*}
whilst the $B$-loops act as follows
\begin{align}
    B_{12} \cdot n(i,j,k) &= \sum_{a,b \in \{-1,1\}} D_{a,b}(i,j,k) n(i+a, j+b, k) \label{eq:nact}\\
    B_{13} \cdot n(i,j,k) &= \sum_{a,b \in \{-1,1\}} D_{a,b}(i,k,j) n(i+a, j, k+b) \notag\\
    B_{23} \cdot n(i,j,k) &= \sum_{a,b \in \{-1,1\}} D_{a,b}(j,k,i) n(i, j+a, k+b)\notag
\end{align}
\end{thm}
\begin{proof}
  See \cref{prop:AactionN} and \cref{prop:BactionN}. 
\end{proof}

Before computing actions in the dumbbell basis $m(i,j,k)$, we use the previous theorem to show the following corollary. 
\begin{cor}\label{cor:irreducible}
When the coefficient ring is $\mathbb{Q}(s)$, the skein module $\Sk(\mathcal H_2)$ is irreducible as a module over the skein algebra $\Sk(\Sigma_{2,0})$. 
\end{cor}
\begin{proof}
First, we note that by \cref{lemma:poly}, the skein module $\Sk(\mathcal{H}_2)$ is actually an algebra, and as an algebra it is isomorphic to the polynomial algebra in variables $B_{12},B_{23},B_{13}$. Furthermore, the loops in $\Sigma_{2,0}$ with these labels act by multiplication operators, which shows that $\Sk(\mathcal H_2)$ is generated as a module by $n(0,0,0)$ (which corresponds to $1$ in the polynomial algebra). It therefore suffices to show that if $x \in \Sk(\mathcal H_2)$ is an arbitrary element, then $n(0,0,0)$ is contained in the subspace $\Sk(\Sigma_{2,0}) \cdot x$.

Second, we note that the operators $A_{i}$ act diagonally on $n(i,j,k)$, and that the joint eigenspaces of the $A_i$ are 1-dimensional when $s$ is not a root of unity. This implies that if $n(i,j,k)$ appears with a nonzero coefficient in the expansion of $x$ in the theta basis, then $n(i,j,k) \in \Sk(\Sigma_{2,0}) \cdot x$. Therefore, it suffices to show that if $n(i,j,k)$ is arbitrary, then $n(0,0,0)$ appears with a nonzero coefficient in some element in $\Sk(\Sigma_{2,0})\cdot n(i,j,k) $.

Finally, \cref{lem:adform} shows that any admissible triple $(i,j,k)$ can be written as $(x+d, y+d, x+y)$ for some $x,y,d \geq 0$. If we start with an admissible triple in that form, we first apply $B_{12}^d$ to $n(i,j,k)$ to obtain a sum of basis elements which has a nonzero coefficient of $n(x,y,x+y)$. We then apply $B_{13}^x$ to this sum to obtain a nonzero coefficient of $n(0,y,y)$, and finally apply $B_{23}^y$ to obtain a nonzero coefficient of $n(0,0,0)$ as desired.
\end{proof}

\begin{thm}
\label{maction}
Let \(i,j,k\) be integers such that the triples \((i, i, j)\) and \((k, k, j)\) are admissible. The $X$-loop and two of the $A$-loops act on the dumbbell basis of $\Sk(\mathcal{H}_2)$ by scalars
\begin{align*}
X \cdot m(i, j, k) & = (-s^{2j+2} - s^{-2j-2}) m(i, j,k ) \notag\\
A_1 \cdot m(i, j, k) &=  (-s^{2i+2} - s^{-2i-2}) m(i, j,k )\notag \\
A_3 \cdot m(i, j, k) &= (-s^{2k+2} - s^{-2k-2}) m(i, j,k ) \notag
\end{align*}
whilst the $B$-loops $B_{12}$ and $B_{23}$ act as follows
\begin{align*}
B_{12} \cdot m(i, j, k) &= {
	\begin{cases} m(i+1, j, k) & \text{ for } i = \frac{j}{2} \\
	 m(i+1, j, k) +\frac{\qint{2i+j+1}\qint{i-j/2}}{\qint{i}\qint{i+1}} m(i-1,j,k) & \text{ for } i> j/2 
\end{cases}} \label{eq:mact} \\
B_{23} \cdot m(i, j, k) &= {
	\begin{cases} m(i, j, k+1) & \text{ for } k = \frac{j}{2} \\
	 m(i, j, k+1) +\frac{\qint{2k+j+1}\qint{k-j/2}}{\qint{k}\qint{k+1}} m(i,j,k-1) & \text{ for } k> j/2 
\end{cases}}
\end{align*}
The action of the middle $A$-loop $A_2$ on the $m$-basis of $\Sk(\mathcal{H}_2)$ is more complex:
\begin{alignat*}{3}
    A_2 \cdot m(i,j,k)= &&{}-\bigg((s^2-s^{-2})^2 \frac{\qint{j/2}^4\qint{i+j/2+1}\qint{k + j/2+1}}{\qint{j-1}\qint{j}^2\qint{j+1}}\bigg) &m(i, j-2, k) \\
    &&-\bigg((s^2-s^{-2})^2\qint{i- j/2}\qint{k - j/2}\bigg)&m(i, j+2, k)\\
    &&+\bigg(\big(- s^{-2(i+k+1)} - s^{2(i+k+1)}\big) + K\bigg) &m(i,j,k)
\end{alignat*}
where 
\[K = (s^2-s^{-2})^2\left(\frac{\qint{j/2 +1}^2\qint{i- j/2}\qint{k -j/2}}{\qint{j+1}\qint{j+2}} +  \frac{\qint{j/2}^2\qint{i+j/2+1}\qint{k + j/2+1}}{\qint{j}\qint{j+1}}\right)\]
\end{thm}
\begin{proof}
  See \cref{prop:AXactionM}, \cref{prop:BactionM} and \cref{thm:appendix}.
\end{proof}

\begin{rmk}\label{rmk:coeff}
In the above expressions, we have used the convention that if $j=0$, then any coefficient term with $[j]$ in the denominator is $0$. In fact, this follows by inspection of these coefficients, since the term $[j/2]^2_s/[j]_s$ (and hence each coefficient itself) is equal to zero when $j=0$.
\end{rmk}
\section{The Handlebody and Representations of DAHAs}\label{sec:Handdaha}
\label{sec:rep}
There are two inclusions of \(\Sigma_{1,1}\)  into \(\Sigma_{2,0}\)  (which we call left and right respectively) and an inclusion of \(\Sigma_{0,4}\) into \(\Sigma_{2,0}\) which are depicted in \cref{fig:embeddings}.
\begin{figure}[!h]
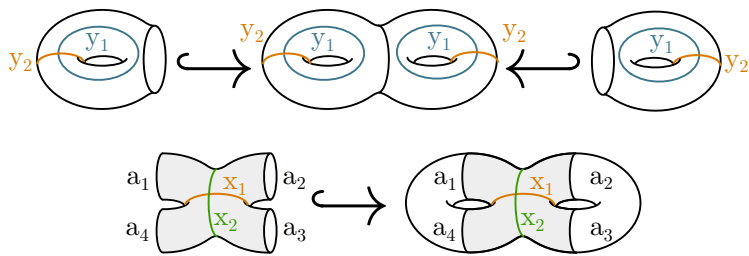

    \centering
    \diagramhh{figures}{surfaceemb}{0pt}{0pt}{0.3}
    \caption{Surface embeddings}\label{fig:embeddings}
\end{figure}

These inclusions induce actions of \(\Sk(\Sigma_{1,1})\) and \(\Sk(\Sigma_{0,4})\) on \(\Sk(\mathcal{H}_2)\). The action of the generators of \(\Sk(\Sigma_{1,1})\) and \(\Sk(\Sigma_{0,4})\) on the \(m(i, j, k)\) basis of \(\Sk(\mathcal{H}_2)\) has already been computed in \cref{sec:calculations}.
In this section we use this to describe how the skein module of the genus 2 handlebody decomposes as a module in terms of the polynomial representations of the $A_1$ and $(C^\vee_1, C_1)$ spherical DAHAs.

\subsection{Module isomorphisms}
In this subsection we provide a useful technical lemma for constructing maps between certain modules over spherical DAHAs. Suppose $B$ is an algebra over a commutative ring $R'$. Suppose $B$ is generated by elements $b_1,b_2,b_3$ that satisfy  three relations of the following form:
\begin{equation}\label{eq:genrels}
[b_i, b_{i+1}]_q = c_{i+2} b_{i+2} + z_{i+2}
\end{equation}
where the indices are taken modulo 3, and where $c_i$ and $z_i$ are elements of $R'$, with the $c_i$ invertible
(Note that we are not assuming that this is a presentation of $B$ as an algebra, only that these three relations hold).

\begin{lem}\label{lem:isos}
Suppose $M$ and $N$ are modules over $B$, and that as modules over the subalgebra $R'[b_1]$, they are generated by $m_0 \in M $ and $n_0 \in N$, respectively. Suppose furthermore that $M$ is free as a module over $R'[b_1]$. Then the assignment $m_0 \mapsto n_0$ extends uniquely to a $R'[b_1]$-linear surjection
\[
\varphi: M \twoheadrightarrow N
\]
Furthermore, suppose that for some $\alpha, \beta, c \in R'$ we have the following identities:
\begin{alignat}{5}
    b_2    &\cdot m_0 &\;=\;& \alpha \, m_0 &\qquad b_2b_1 &\cdot m_0 &\;=\;& \beta a_1 m_0 + c\label{eq:basecase}\\
    b_2    &\cdot n_0 &=\;& \alpha \, n_0 \notag &\qquad b_2b_1 &\cdot n_0 &=\;& \beta a_1 n_0 + c\notag
\end{alignat}
Then $\varphi$ is a map of $B$-modules.
\end{lem}
\begin{rmk}
Before we prove this technical lemma, let us explain how it will be used. The skein algebras of the 4-punctured sphere and once-punctured torus map to spherical double affine Hecke algebras, and this allows us to restrict the `polynomial representations' of the spherical DAHAs to modules over these skein algebras. Both these skein algebras are generated by 3 elements that satisfy relations of the form \eqref{eq:genrels}. 
We will use this lemma to construct maps from the polynomial representations coming from spherical DAHAs to the skein module of the genus 2 handlebody. 
\end{rmk}
\begin{proof}
Since $M$  is free of rank 1 over $R'[b_1]$, the claimed surjection exists and  is uniquely defined by $\varphi(f(b_1)m_0) = f(b_1)n_0$. By construction $\varphi$ is a map of $R'[b_1]$-modules.  What remains to be shown is that this map commutes with the action of $b_2$ and $b_3$. We will prove this by induction on the degree of $f(b_1)$. For the base case $f(b_1) = 1$, the required commutativity for $b_2$ follows from the first two equations in \eqref{eq:basecase}. For $b_3$, the relation \eqref{eq:genrels} shows $b_3 = c_3^{-1} [b_1,b_2]_q - c_3^{-1}z_3$, and then commutativity follows from the second two equations of \eqref{eq:basecase}.

For the inductive step, assume $ \varphi(b_i f(b_1)m_0) = b_i f(b_1) n_0$ for all $f(b_1)$ of degree at most $n$. We want to show that $\varphi(b_i b_1^{n+1} m_0) = b_i b_1^{n+1} n_0$ for $i=2,3$. We have
\begin{align*}
    \varphi(b_2 b_1^{n+1} m_0) &= \varphi((b_2 b_1) b_1^n)\\
    &= \varphi((q^2b_1 b_2 - qc_3 - q z_3) b_1^n m_0) \\
    &= (q^2 b_1 b_2 - qc_3 - q z_3)b_1^n n_0\\
    &= b_2 b_1^{n+1} n_0
\end{align*}
where the step from the second to third line follows from the induction hypothesis and $R'[b_1]$-linearity of $\varphi$. This proves commutativity for $b_2$, and the proof for $b_3$ is similar.
\end{proof}

\subsection{The action of \(\Sk(\Sigma_{1,1})\) on \(\Sk(\mathcal{H}_2)\)}
In this section we first consider the action \(\cdot = \cdot_L\) as \(\cdot_R\) is analogous.
From \cref{thm:torusskeinpres} we know that the skein algebra \(\Sk(\Sigma_{1,1})\) is generated by the loops \(y_1, y_2, y_3\), but as \[y_3 = \frac{[y_1, y_{2}]_{s}}{(s^2 - s^{-2})}\] it is sufficient to determine the action of \(y_1\) and \(y_2\). These actions were computed in previous sections, and we recall the results in the present notation.
\begin{align}
    y_2 \cdot m(i, j, k) &= (-s^{2i+2} - s^{-2i-2}) m(i,j,k) \label{eq:sktaction}\\
    y_1 \cdot m(i, j, k) &= {
	\begin{cases} m(i+1, j, k) & \text{ for } i = \frac{j}{2} \\
	 m(i+1, j, k) +\frac{\qint{2i+j+1}\qint{i-\frac{j}{2}}}{\qint{i}\qint{i+1}} m(i-1,j,k) & \text{ for } i> j/2 \end{cases}} \notag
\end{align}

Let \(V = \Sk( \mathcal{H}_2)\) with basis \(\{m(i,j,k)\}\), so
\[
    \Sk(\mathcal H_2) =: V = \bigoplus_{j \text{ is even }} \bigoplus_{i,k = j/2}^{\infty} \rspan \{ m(i,j,k)\} 
\]
Let \(V_{j,k}\) denote the subspace with fixed \(j\) and \(k\), and $V_j$ denote the subspace with fixed $j$, so
\[
    V_{j,k} = \bigoplus_{i = j/2}^{\infty} \rspan \{ m(i,j,k)\} ,\quad \quad V_{j} = \bigoplus_{i=j/2}^\infty \bigoplus_{k=j/2}^\infty \rspan \{ m(i,j,k)\} 
\]
The subspace $V_{j,k}$ is invariant under the $\cdot_L$ action of \(y_1\) and \(y_2\), and $V_j$ is invariant under both the $\cdot_L$ and $\cdot_R$ actions. Therefore there are representations
\begin{align*}
\rho_{j,k}:& \Sk(\Sigma_{1,1}) \to \End_R(V_{j,k})\\
\rho_j:& \Sk(\Sigma_{1,1}) \otimes_R \Sk(\Sigma_{1,1}) \to \End_R(V_{j})
\end{align*}
defined by this action. We first study the representation $\rho_{j,k}$ and then use this to study the representation $ \rho_j$.

The submodule \(V_{j,k}\) is free of rank \(1\) as a module over \(R[y_1]\) with generator \(m_0 := m(j/2, j, k)\). Furthermore, using equation \eqref{eq:sktaction}, we see that
\begin{align}
y_2 \cdot m_0 &= (-s^{j+2} - s^{-j-2}) m_0\notag\\
y_2 y_1 \cdot m_0 &=
(-s^{2(j/2 + 1)+2} - s^{-2(j/2 + 1)-2}) y_1 \cdot m_0 \label{eq:skactreminder}\\
X \cdot m(i, j, k) & = (-s^{2j+2} - s^{-2j-2}) m(i, j,k )\notag
\end{align}

Recall that the $A_1$ spherical DAHA acts on the polynomial representation $P(q,t) = R[x]$. We can restrict this action to the skein algebra using the algebra map $\Sk(\Sigma_{1,1}) \to SH_{q,t}$ of \cref{prop:sktodahaA}. We recall that in the current notation, this map is uniquely determined  by the following assignments:
\begin{equation}\label{eq:isoreminder}
y_1 \mapsto x,\quad \quad y_2 \mapsto y,\quad \quad X \mapsto -q^2 t^{-2} - q^{-2} t^2.
\end{equation}
\begin{thm}\label{thm:a1decomp}
The $\Sk(\Sigma_{1,1})$-module $V$ is isomorphic to the direct sum
\begin{equation}\label{eq:ptdecomp}
\bigoplus_{\substack{j,k \geq 0,\\ j \textrm{ even}}} P(q=s, t=-s^{-j-2})
\end{equation}
\end{thm}
\begin{proof}
By equation \eqref{cor:a1dprep}, in the polynomial representation $P(q,t)$ we have
\begin{align*}
    y\cdot 1 &= t+t^{-1}\\
    yx\cdot 1 &= (q^2 t^{-1} + q^{-2}t)x
\end{align*}
When we specialise $q=s$ and $t=-q^{-j-2}$, we see that these formulas are compatible with the formulas \eqref{eq:skactreminder} under the assignments \eqref{eq:isoreminder}, when we send $1 \in P(q,t) = R[x]$ to $m_0 \in V_{j,k}$. Then \cref{lem:isos} shows
\[
P(q=s, t=-q^{-j-2}) \cong V_{j,k}
\]
\end{proof}
In the previous theorem we only decomposed $V$ as a module over the action of one copy of the DAHA, which corresponds to the embedded punctured torus on the left of \cref{fig:embeddings}. This is the reason that the submodules in the decomposition in \eqref{eq:ptdecomp} doesn't depend on $k$ (Conversely, under the (left) action of the right copy of $\Sk(\Sigma_{1,1})$, the $k$-indices vary and the $i$ indices do not). The actions of the skein algebras of the `left' and `right' punctured tori commute with each other, and we can decompose $V$ as a module over the tensor product of these two algebras as follows.

Let $B = \Sk(\Sigma_{1,1}) \otimes_R \Sk(\Sigma_{1,1})$, which acts on $V$ as described in the previous paragraph. If $M$ and $N$ are modules over the $A_1$ DAHA, then $B$ acts on $M \otimes_R N$ via restriction of the (tensor square of) the algebra map $\Sk(\Sigma_{1,1}) \to SH_{q,t}$.

\begin{thm}\label{thm:2a1decomp}
The $B$-module $V$ has the following decomposition
\begin{equation}\label{eq:tptdecomp}
V \cong \bigoplus_{\substack{j \geq 0,\\ j \textrm{ even}}} P(q=s, t=-s^{-j-2}) \otimes_R P(q=s, t=-s^{-j-2})
\end{equation}
\end{thm}
\begin{proof}
The proof is essentially the same as the proof of \cref{thm:a1decomp}, and it proceeds by showing
\[
V_j \cong P(q=s, t=-s^{-j-2}) \otimes_R P(q=s, t=-s^{-j-2})
\]
This follows from \cref{thm:a1decomp} and from the `tensor square' of \cref{lem:isos}.
\end{proof}

\subsection{The action of \(\Sk(\Sigma_{0,4})\) on \(\Sk (\mathcal{H}_2)\)}
In this section we study the action of the skein algebra of the 4-punctured sphere on the skein module $\Sk(\mathcal H_2)$ of the genus 2 handlebody. To shorten notation, we write
\[
V := \Sk(\mathcal H_2).
\]

From \cref{thm:presentationofsphereskein} we know that the skein algebra \(\Sk(\Sigma_{0,4})\) is generated by \(x_1, x_2, x_3, a_1, a_2, a_3, a_4\) where $x_i$ is the curve which separates the punctures $a_i,a_{i+1}$ from the punctures $a_{i+2}, a_{i+3}$ (with indices taken modulo 4). 
The element $x_3$ can be written in terms of the other generators,
so it suffices to compute the actions of \(x_1, x_2, a_1, a_2, a_3\) and \(a_4\).  These computations were done in \cref{maction}, and we recall the results here in the present notation.
For all \(i,j,k\) such that the triples \((i, i, j)\) and \((k, k, j)\) are admissible,
\begin{align}
a_1 \cdot m(i, j, k) &= a_4 \cdot m(i, j, k) = (-s^{2i+2} - s^{-2i-2}) m(i, j,k ), \label{eq:a14act}\\
a_2 \cdot m(i, j, k) &= a_3 \cdot m(i, j, k) = (-s^{2k+2} - s^{-2k-2}) m(i, j,k ), \label{eq:a23act} \\
x_2 \cdot m(i, j, k) & = (-s^{2j+2} - s^{-2j-2}) m(i, j,k ). \label{eq:x2act}
\end{align}
Finally,
\begin{alignat}{3}
    x_1 \cdot m(i,j,k)= &&{}-\bigg((s^2-s^{-2})^2 \frac{\qint{j/2}^4\qint{i+j/2+1}\qint{k + j/2+1}}{\qint{j-1}\qint{j}^2\qint{j+1}}\bigg) &m(i, j-2, k) \label{eq:x1act}\\
    &&-\bigg((s^2-s^{-2})^2\qint{i- j/2}\qint{k - j/2}\bigg)&m(i, j+2, k)\notag \\
    &&+\bigg(\big(- s^{-2(i+k+1)} - s^{2(i+k+1)}\big) + K\bigg) &m(i,j,k)\notag 
\end{alignat}
where $K$ is a certain constant.
Recall that in these formulas, if a coefficient has $[j]_s$ in the denominator, then that coefficient is equal to $0$ when $j=0$ (see \cref{rmk:coeff}).

Now define the following $R$-submodules of $V$:
\begin{equation}
    V_{i,k} := \rspan \{m(i,2j,k) \mid 0 \leq j \leq  i, k \} \label{eq:submods}
\end{equation}
We note that these $R$ modules have $R$-dimension $1 + min(i,k)$.
\begin{lem}\label{lem:s4decomp}
Over the skein algebra $\Sk(\Sigma_{0,4})$, the module $V$ has the following decomposition:
\[
V = \bigoplus_{i,k \geq 0} V_{i,k}
\]
\end{lem}
\begin{proof}
If $\mathcal A \subset \mathbb N^{\times 3}$ is the set of triples $(i,j,k)$ with $(i,i,j)$ and $(j,k,k)$ admissible, then we already know the set $\{m(i,j,k) \,\mid\, (i,j,k) \in \mathcal A\}$ is an $R$-basis for $V$. From the definition of admissibility (\cref{def:adm}), it is clear that each $m(i,j,k)$ is an element of exactly one of the $V_{i,k}$, which shows the claimed decomposition as $R$-modules. All that is left is to show that each $V_{i,k}$ is preserved under the action of $\Sk(\Sigma_{0,4})$. The $a_i$ and $x_2$ preserve this decomposition since they act diagonally on the $m(i,j,k)$, and since the skein algebra is generated by these elements and $x_1$, all that remains is to show that $x_1$ preserves this decomposition. This follows from the formula for the action of $x_1$ above, since the operator $x_1$ doesn't change $i$ or $k$, and the coefficient of $m(i,j+2,k)$ is equal to 0 whenever $(i,i,j)$ or $(j,k,k)$ is not admissible (since $[0]_q = 0$).
\end{proof}

We now identify the pieces $V_{i,k}$ as finite dimensional quotients of polynomial representations of DAHAs. More precisely, by \cref{prop:Bereset&SamuelsonIso} (which we recall below), there is a map from the skein algebra of the 4-punctured sphere to the $(C^\vee_1, C_1)$ spherical DAHA, and this algebra acts on $\mathbb C[x]$ as described in \cref{sec:polyrep}. By the computations in that section, this module is generated by $1$ and is free over the subalgebra $\mathbb C[x]$ (where $x$ acts by multiplication). We recall from \cref{lem:bc1dahaact} that we have the following identities:
\begin{align}
    y\cdot 1 &= t_1 t_3 + t_1^{-1} t_3^{-1}\label{eq:yact}\\
    yx \cdot 1 &= (q^2 t_1^{-1} t_3^{-1} + q^{-2} t_1 t_3) x + -\ti \tiv + \tiv (q^{-2} t_1 -q^2 t_1^{-1}) + (q^{-1}-q)\tii (t_3 + t_3^{-1})\label{eq:yxact}
\end{align}
Let $P_{i,k}$ be this polynomial representation with parameters specialised as follows:
\begin{equation}\label{eq:params}
q=s^2,\quad \quad t_1 = \iota s^{-2i-2},\quad \quad t_2 = \iota s^{\epsilon_2(2k+2)}, \quad \quad t_3 = \iota s^{2i},\quad \quad t_4 = \iota s^{\epsilon_4 (2k+2)}
\end{equation}
Here $\iota^2 = -1$, and $\epsilon_2$ and $\epsilon_4$ are signs which can be chosen arbitrarily (the choice does not affect the computations below). We view $P_{i,k}$ as a module over the skein algebra $\Sk(\Sigma_{0,4})$ by restriction along the algebra map \(\phi: \Sk(\Sigma_{0,4}) \to S\mathscr{H}_{q,\underline{t}}\)  described in \cref{prop:Bereset&SamuelsonIso}. 
This algebra map $\phi$ is determined by the assignments
\begin{alignat*}{3}
    \phi(x_1) &= x,&\quad& \phi(a_1) &\;=& \;\iota  \ti,\\
    \phi(x_2) &= y,&& \phi(a_2) &=&\;\iota \tii,\\
    \phi(x_3) &= z,&& \phi(a_3) &=&\;\iota  \tiv,\\
   \phi(s) &= q^2, && \phi(a_4) &=&\;\iota  (\tiii).
\end{alignat*}

Let $R' = R[a_1,a_2,a_3,a_4]$. Now we  construct the following map of $\Sk(\Sigma_{0,4})$-modules: 
\begin{equation}\label{eq:4pmap}
    \varphi: P_{i,k} \twoheadrightarrow V_{i,k}
\end{equation} 
This map is uniquely determined by the choice $1 \mapsto m(i,0,k)$ and by the requirement that it is a module map over $R[x_1]$ (since $P_{i,k}$ is freely generated by $1$ over $R[x_1]$). To see that it is a map of $R'[x_1]$-modules, we need to show that $\varphi(a_i\cdot m) = a_i\cdot \varphi(m)$ for any $m \in P_{i,k}$. The $a_i$ act as scalars in $P_{i,k}$ by definition, and we can see that they act as scalars in $V_{i,k}$ from the formulas \eqref{eq:a14act} and \eqref{eq:a23act} since these actions just depend on $i$ and $k$, and these indices are the same for all basis vectors inside $V_{i,k}$. All that remains is to show that these scalars agree, and this follows immediately from the specializations in \eqref{eq:params}.
\begin{lem}
The map $\varphi$ of equation \eqref{eq:4pmap} is a map of $\Sk(\Sigma_{0,4})$-modules.
\end{lem}
\begin{proof}
By \cref{lem:isos}, it suffices to check the following identities:
\begin{align}
    \varphi(x_2\cdot 1) &= x_2\cdot m(i,0,k)\label{eq:need1} \\
    \varphi(x_2 x_1 \cdot 1) &= x_2 x_1 \cdot m(i,0,k)\label{eq:need2}
\end{align}
By equation \eqref{eq:yact} and the specialisations in \eqref{eq:params}, we see that
\[
x_2\cdot 1 = t_1 t_3 + t_1^{-1}t_3^{-1} = \iota^2 (s^{-2i-2}s^{2i} + s^{2i+2}s^{-2i}) = -s^2-s^{-2}
\]
Combining this with \eqref{eq:x2act} proves \eqref{eq:need1}.

Let $m_2 := m(i,2,k)$ and $m_0 := m(i,0,k)$. To prove equation \eqref{eq:need2}, we first note that equations \eqref{eq:x1act} and \eqref{eq:x2act} show the identity
\begin{align*}
x_2 x_1 m_0 &= c_2 c m_2 + c_0 d m_0\\
&= c_2( cm_2 + dm_0) + (-c_2 d + c_0 d)m_0\\
&= c_2 (x_1 m_0) + (-c_2  + c_0 )d m_0
\end{align*}
where we have used the constants $c_i := -s^{2i+2}-s^{-2i-2}$ and 
\[
c= -(s^2-s^{-2})^2 [i]_s[k]_s,\quad \quad 
d = -s^{-2i-2k-2}-s^{2i+2k+2} + (s^2-s^{-2})^2 \frac{[i]_s[k]_s}{[2]_s}
\]
Similarly, equation \eqref{eq:yxact} shows that $x_2x_1 \cdot 1 = ax_1 + b$, where
\[
a = q^2 t_1^{-1} t_3^{-1} + q^{-2} t_1 t_3,\quad  b =  -\ti \tiv + \tiv (q^{-2} t_1 -q^2 t_1^{-1}) + (q^{-1}-q)\tii (t_3 + t_3^{-1})
\]
All that remains is to show that with the specialisations of \eqref{eq:params} we have
\[
a = c_2,\quad b = (-c_2 + c_0)d
\]
and these are straightforward identities of Laurent polynomials in the parameter $s$.
\end{proof}

\begin{rmk}
Oblomkov and Stoica completely classified finite dimensional representations of the $(C^\vee_1, C_1)$ DAHA in \cite{OS09} and showed that they are all quotients of the polynomial representation (possibly with some different signs) at certain special parameter values. They also show that at these parameter values, the polynomial representation has a \emph{unique} nontrivial quotient, which leads to the following corollary.
\end{rmk}

\begin{cor}\label{cor:ccdecomp}
The skein module of the genus 2 handlebody is a direct sum of finite dimensional representations of the spherical $(C^\vee_1, C_1)$ DAHA. Each summand is the unique quotient of the polynomial representation with parameters specialised as in \eqref{eq:params}, with $i,k \geq 0$.
\end{cor}

\subsection{Leonard pairs}
Several years ago Terwilliger asked the second author whether Leonard pairs appeared in skein theory, and in this section we show that the answer is yes.

\begin{defn}
A matrix with entries $a_{ij}$ is \emph{irreducible tridiagonal} when the only nonzero entries are $a_{i j}$ with $\lvert i - j \rvert \leq 1$, and  $a_{i j} \not= 0$ whenever $\lvert i - j \rvert = 1$.
\end{defn}

\begin{defn}
A \emph{Leonard pair} is a finite dimensional vector space $V$ equipped with endomorphisms $A$ and $B$ which satisfy the following conditions:
\begin{enumerate}
    \item There exists an $A$-eigenbasis of $V$ for which the matrix representing $B$ is irreducible tridiagonal.
    \item There exists a $B$-eigenbasis of $V$ for which the matrix representing $A$ is irreducible tridiagonal.
\end{enumerate}
\end{defn}

Recall from \eqref{eq:submods} that $V_{i, k}$ is a finite-dimensional subspace of the skein module of the handlebody, and it is a submodule over the skein algebra $\Sk(\Sigma_{0,4})$. In particular, $x_1$ and $x_2$ act on $V_{i,k}$.

\begin{cor}\label{cor:leonard}
The vector space $V_{i, k}$ with the endomorphisms $x_1$ and $x_2$ is a Leonard pair.
\end{cor}
\begin{proof}
First, by \cref{thm:presentationofsphereskein} there is a map from Terwilliger's universal Askey-Wilson algebra $\Delta_q$ to $\Sk(\Sigma_{0,4})$ which takes the generators $A,B$ to $x_1,x_2$. By \cite{Ter18}, there is a map from the $q$-Onsager algebra $\mathcal O_q$ to $\Delta_q$, and hence to $\Sk(\Sigma_{0,4})$. This implies that $V_{i, k}$ is a module over $\mathcal O_q$, and by \cref{lemma:vikirr}, it is irreducible over $\mathcal O_q$. 

The $q$-Onsager algebra is a specialisation of the tridiagonal algebra of \cite{Ter01}. Since $V_{i,k}$ is irreducible and finite dimensional,  \cite[Thm. 3.10]{Ter01} implies the action of $x_1$ and $x_2$ on $V_{i,k}$ give a tridiagonal pair. Since we know $x_1$ and $x_2$ have 1-dimensional eigenspaces, \cite[Lem. 2.2]{Ter01} implies $x_1,x_2$ is a Leonard pair on $V_{i,k}$.
\end{proof}

\begin{rmk}
Instead of the proof above, we could compute the action of $x_2$ on the $n(i,j,k)$ basis to check that it is tridiagonal. This would be similar in difficulty to the computation of the coefficients of the $x_1$ action in the $m(i,j,k)$ basis.
\end{rmk}

Using our explicit formulas for the action of $x_1$ and $x_2$, we show the following.
\begin{lem}\label{lemma:vikirr}
With the base ring $\mathbb Q(s)$, the module $V_{i,k}$ is irreducible as a module over the subalgebra of $\Sk(\Sigma_{0,4})$ generated by $x_1$ and $x_2$.
\end{lem}
\begin{proof}
Write $A$ for the subalgebra of the skein algebra generated by $x_1$ and $x_2$. We would like to show that if $x \in V_{i,k}$, then $A\cdot x = V_{i,k}$. Equation \eqref{eq:x2act} shows that $x_2$ acts on $V_{i,k}$ with 1-dimensional eigenspaces. This implies that it is sufficient to show the following implication: for any $j$, and for each $j'$ with $(i,i,j')$ and $(j',k,k)$ admissible, the subspace $A\cdot m(i,j,k)$ contains an element $y$ whose $m(i,j,k)$-expansion has a nonzero coefficient for $m(i,j',k)$. Then equation \eqref{eq:x1act} shows that we may take $y = x_1^{\lvert j - j'\rvert} \cdot m(i,j,k)$.
\end{proof} 
\section{The Arthamonov-Shakirov Genus Two  DAHA}
\label{ASlink}
In \cite{arthamonov&shakirov2017}, Arthamonov and Shakirov propose a definition for a genus $2$, type $A_1$, spherical DAHA. 
In this section we show that when $q=t$ this algebra is isomorphic to the skein algebra $\Sk(\Sigma_{2,0})$ (see \cref{thm:AScompare}).

The \(A_1\) spherical DAHA admits a  representation on $\mathbb{Q}(q,t)[X^{\pm 1}]$ as the algebra generated by two operators
\[\hat{\mathcal{O}}_B := X + X^{-1}, \quad \hat{\mathcal{O}}_A := \frac{X^{-1} - tX}{X^{-1} - X} \hat{\delta}  + \frac{X - t X^{-1}}{X - X^{-1}} \hat{\delta}^{-1}\]
where \(\hat{\delta} f(X) := f(q^{1/2} X)\) is the shift operator; these two operators are associated with the A- and B-cycles of the torus. In order to generalise the $A_1$ spherical DAHA to genus $2$, Arthamonov and Shakirov generalise this representation: they use six operators 
\( \hat{\mathcal{O}}_{B_{1 2}}\), 
\(\hat{\mathcal{O}}_{B_{1 3}}\), 
\(\hat{\mathcal{O}}_{B_{2 3}}\), 
\(\hat{\mathcal{O}}_{A_1}\), 
\(\hat{\mathcal{O}}_{A_2}\), 
\(\hat{\mathcal{O}}_{A_3}\), 
which relate to the six cycles shown in  \cref{fig:cycles}, and these operators act on \(\mathcal{H}\), the space of Laurent polynomials in variables \(x_{12}\), \(x_{13}\), \(x_{23}\) which are symmetric under the \(\mathbb{Z}^3_2\) group of Weyl inversions \((x_{12}, x_{13}, x_{23})  \mapsto  (x_{12}^u, x_{13}^v, x_{23}^w)\) where \(\quad u, v, w \in \{\pm 1\}\).
This space \(\mathcal{H}\) has a basis given by a family of Laurent polynomials \(\{\Psi_{i, j, k}\}\) where \((i, j, k)\) are admissible triples \cite[Def. 2]{arthamonov&shakirov2017}.
We shall show that, when \(q = t = s^4\), the algebra of actions of $\Sk(\Sigma_{2,0})$ on \(\Sk(\mathcal{H})\) is isomorphic to this algebra of six operators via the mapping 
\[
    \lp \mapsto - \hat{\mathcal{O}}_{\lp} \text{ and } n(i, j, k) \mapsto \alpha^{-1}(i,j,k) \Psi_{i, j, k}
\]
where \(\alpha: \Ad \to \mathbb{Q}(q^{\pm\frac{1}{2}})\) is a change of basis map which we define in \cref{defn:alpha} and give a closed form for in \cref{prop:alpha_non_recursive}.

We shall first recall the necessary definitions from \cite{arthamonov&shakirov2017}.
\begin{defn}{}
    For any admissible triple \((i,j,k) \in \Ad\) and \(a,b \in \{-1,1\}\), the \emph{Arthamonov and Shakirov coefficients} are
    \[
    C_{a,b} (i,j,k) = a b \frac{\left[\frac{ai + bj + k}{2}, \frac{a+b+2}{2}\right]_{q,t} \left[\frac{ai + bj - k}{2}, \frac{a+b}{2}\right]_{q,t}  [i-1, 2]_{q,t}[j-1, 2]_{q,t}}{ \left[i, \frac{a+3}{2}\right]_{q,t} \left[i-1, \frac{a+3}{2}\right]_{q,t} \left[j, \frac{b+3}{2}\right]_{q,t} \left[j-1, \frac{b+3}{2}\right]_{q,t} }
    \]
    where
    \[
    [n,m]_{q,t} := \frac{q^{\frac{n}{2}} t^{\frac{m}{2}} - q^{-\frac{n}{2}} t^{-\frac{m}{2}} }{ q^{\frac{1}{2}} - q^{-\frac{1}{2}}}.
    \]
\end{defn}

\begin{defn}
    Given a loop in $\{A_1, A_2, A_3, B_{1 2}, B_{2 3}, B_{1 3} \}$, the operator $\hat{\mathcal{O}}_{\lp}$ acts on $\Psi_{i,j,k}$ as follows:
    \begin{alignat*}{5}
         \hat{\mathcal{O}}_{B_{1 2}}\Psi_{i,j,k} &= \sum_{a,b \in \{-1,1\}} C_{a,b}(i,j,k) \Psi_{i+a, j+b, k} 
         &\qquad& \hat{\mathcal{O}}_{A_{1}} \Psi_{i, j, k} &= \Bigl( q^{i/2} t^{1/2} + q^{-i/2} t^{-1/2} \Bigr) \Psi_{i, j, k} \\
        \hat{\mathcal{O}}_{B_{1 3}}\Psi_{i,j,k} &= \sum_{a,b \in \{-1,1\}} C_{a,b}(i,k,j) \Psi_{i+a, j, k+b}
        &&\hat{\mathcal{O}}_{A_{2}} \Psi_{i, j, k} &= \Bigl(q^{j/2} t^{1/2} + q^{-j/2} t^{-1/2}\Bigr) \Psi_{i, j, k} \\
        \hat{\mathcal{O}}_{B_{2 3}}\Psi_{i,j,k} &= \sum_{a,b \in \{-1,1\}} C_{a,b}(j,k,i) \Psi_{i, j+a, k+b}
        && \hat{\mathcal{O}}_{A_{3}} \Psi_{i, j, k} &= \Bigl(q^{k/2} t^{1/2} + q^{-k/2} t^{-1/2} \Bigr) \Psi_{i, j, k}
        \end{alignat*}
    where \(\Psi_{0,0,0} = 1\) and \(\Psi_{i,j,k} := 0\) if \((i,j,k)\) is not an admissible triple. 
    The \emph{Arthamonov--Shakirov, genus $2$, spherical DAHA} is the subalgebra of endomorphism ring of $\mathcal{H}$ generated by these operators. 
\end{defn}

From now on we shall specialise to $q = t = s^4$ so that \([n,m]_{q,t} = \qint{n+m}\) (this makes sense because none of the structure constants above have poles at $t=q$).
This leads to a simplification of the Arthamonov and Shakirov coefficients \(C_{1,1} (i,j,k)\).
\begin{lem}
    \label{C relations}
    For all \((i,j,k) \in \Ad\) we have that
    \begin{equation}
    \label{eq:C11}
    C_{1,1} (i,j,k) = \frac{\qint{\frac{i + j + k + 4}{2}} \qint{\frac{i + j - k + 2}{2}}}{ \qint{i+2}  \qint{j+2}}
    \end{equation}
    and the relations:
    \begin{align}
        C_{1,1} (i,j,k) &= C_{1,1} (j,i,k) \label{ijsymmetric}\\
        C_{1,1}(k,j+1,i+1) C_{1,1} (i,j,k) &= C_{1,1} (i, j+1, k+1) C_{1,1} (k,j,i) \label{path relation}
    \end{align}
\end{lem}
\begin{proof}
Substituting $a=b=1$ into the definition of $C_{a,b}$ gives \eqref{eq:C11}. As this is symmetric in $i$ and $j$ we have \eqref{ijsymmetric} and by substituting \eqref{eq:C11} for each term in \eqref{path relation} we see that this relation holds.
\end{proof}

We also note that any admissible triple can be written in the following form:
\begin{lem}
\label{lem:adform}
A triple is admissible if and only if it is of the form \((x+d, y+d, x+y)\) for some \(x, y, d \geq 0\).
\begin{proof}
The triple \((x+d, y+d, x+y)\) is admissible if $x,y,d \geq 0$. Conversely, if the triple \((a,b,c)\) is admissible, then let $2x = a+c-b$, $2y = b+c-a$, and $2d = a+b-c$. Then $x,y,d \geq 0$ by the admissibility of the triple \(a,b,c\), and \((a,b,c) = (x+d, y+d, x+y)\) by construction.
\end{proof}
\end{lem}

Therefore, in order to define the change of basis map \(\alpha: \Ad \to \mathbb{Q}[s^2, s^{-2}]\), which relates the theta basis of the skein module $\Sk(\mathcal{H}_2)$ to the basis of $\mathcal{H}$ in \cref{thm:AScompare}, it is sufficient to define \(\alpha(x+d, y+d, x+y)\) for $x, y, d \geq 0$.
\begin{defn}
\label{defn:alpha}
Define \(\alpha: \Ad \to \mathbb{Q}[s^2, s^{-2}]\) recursively as follows:
\begin{align}
    \alpha(0,0,0) &= 1 \label{base}\\
    \alpha(i+1, j+1, k) & = - \frac{\alpha(i,j,k)}{C_{1,1}(i,j,k)}; \label{rule1}\\
    \alpha(i+1, j, k+1) & = - \frac{\alpha(i,j,k)}{C_{1,1}(i,k,j)}; \label{rule2}\\
    \alpha(i, j+1, k+1) & = - \frac{\alpha(i,j,k)}{C_{1,1}(j,k,i)}. \label{rule3}
\end{align}
\end{defn}

\begin{lem}
The function \(\alpha: \Ad \to \mathbb{Q}[s^2, s^{-2}]\) is well defined.
\begin{proof}
In order to prove that \(\alpha(i,j,k)\) is unambiguously defined it is sufficient to prove that if \(\alpha(i,j,k)\) is unambiguously defined then \(\alpha(i+1, j+2, k+1)\), \(\alpha(i+2, j+1, k+1)\) and \(\alpha(i+1, j+1, k+2)\) are unambiguously defined.
\begin{align*}
    \alpha(i+1, j+2, k+1)
        &= - \frac{\alpha(i,j+1,k+1)}{C_{1,1}(i,j+1,k+1)} \text{ using \eqref{rule1} first}\\
        &= \frac{\alpha(i,j,k)}{ C_{1,1}(i,j+1,k+1) C_{1,1}(j,k,i)} \\
    \alpha(i+1, j+2, k+1)
        &= - \frac{\alpha(i+1,j+1,k)}{C_{1,1}(j+1,k,i+1)} \text{ using \eqref{rule3} first}\\
        &= \frac{\alpha(i,j,k)}{C_{1,1}(j+1,k,i+1) C_{1,1}(i,j,k)} \\
        &= \frac{\alpha(i,j,k)}{C_{1,1}(k,j+1,i+1) C_{1,1}(i,j,k)} \text{ by  \cref{C relations} \eqref{ijsymmetric}} \\
        &= \frac{\alpha(i,j,k)}{C_{1,1}(i,j+1,k+1) C_{1,1}(k,j,i)} \text{ by \cref{C relations} \eqref{path relation}} \\
        &= \frac{\alpha(i,j,k)}{C_{1,1}(i,j+1,k+1) C_{1,1}(j,k,i)} \text{ by \cref{C relations} \eqref{ijsymmetric}}
\end{align*}
these give the same result.
Checking \(\alpha(i+2, j+1, k+1)\) and \(\alpha(i+1, j+1, k+2)\) are unambiguously defined is analogous.
\end{proof}
\end{lem}

We shall now find a non-recursive formula for \(\alpha\).
\begin{lem}
\label{x only}
For all \(x\geq 0\), \(\alpha(x, 0, x) = (-1)^{x} \qint{x+1}.\)
\begin{proof}
We proceed by induction on $x$: for $x = 0$ the result holds by the definition of $\alpha$ and  
\[\alpha(x+1,0,x+1) = - \frac{\alpha(x,0,x)}{C_{1,1}(x,x,0)} = (-1)^{x+1} \frac{\qint{x+1}\qint{x+2}}{\qint{x+1}} = (-1)^{x+1} \qint{x+2}\]
by the induction assumption and \eqref{eq:C11}.
\end{proof}
\end{lem}

\begin{lem}
\label{x and y}
For all \(x,y \geq 0\), \(\alpha(x, y, x+y) = (-1)^{x+y} \qint{x+1} \qint{y+1}\).
\begin{proof}
We proceed by induction on $y$: for $y = 0$ this is \cref{x only} and  
\[\alpha(x+1,0,x+1) = - \frac{\alpha(x,0,x)}{C_{1,1}(x,x,0)} = (-1)^{x+1} \frac{\qint{x+1}\qint{x+2}}{\qint{x+1}} = (-1)^{x+1} \qint{x+2}\]
by the induction assumption and \eqref{eq:C11}.
\end{proof}
\end{lem}

\begin{prop}
\label{prop:alpha_non_recursive}
For all \(x,y,d \geq 0\),
\[\alpha(x+d, y+d, x+y) = (-1)^{x+y+d} \frac{\qint{x+1} ... \qint{x+d+1} \qint{y+1} ... \qint{y+d+1}}{\qint{1} ...\qint{d} \qint{x+y+2} ... \qint{x+y+d+1}}\]
As all admissible triples can be written in this form this means that this is an equivalent definition for \(\alpha\).
\begin{proof}
Note that using \cref{C relations} we have that
\[
C_{1,1}(x+d,y+d,x+y)  = \frac{\qint{x+y+d+2} \qint{d + 1}}{ \qint{x+d+2}  \qint{y+d+2}}
\]
We shall prove this result by induction on \(d\). The base case holds by \cref{x and y}, and as
\begin{align*}
&\alpha(x+d+1, y+d+1, x+y) \\
    &= - \frac{\alpha(x+d,y+d,x+y)}{C_{1,1}(x+d,y+d,x+y)} \\
    &= \left( (-1)^{x+y+d+1} \frac{\qint{x+1} ... \qint{x+d+1} \qint{y+1} ... \qint{y+d+1}}{\qint{1} ...\qint{d} \qint{x+y+2} ... \qint{x+y+d+1}}\right) \left( \frac{ \qint{x+d+2}  \qint{y+d+2}}{\qint{x+y+d+2} \qint{d + 1}} \right)\\
    &= (-1)^{x+y+d+1} \frac{\qint{x+1} ... \qint{x+d+1} \qint{x+d+2} \qint{y+1} ... \qint{y+d+1} \qint{y+d+2}}{\qint{1} ...\qint{d}\qint{d+1} \qint{x+y+2} ... \qint{x+y+d+1}\qint{x+y+d+2}}
\end{align*}
we have proven the induction step.
\end{proof}
\end{prop}

\begin{thm}\label{thm:AScompare}
    For the specialisation \(q=t=s^4\), the action of $\hat{\mathcal{O}}_{\lp}$ on \(\Psi(i,j,k) \in \mathcal{H}\) is equivalent to the action of \(\lp\) on \(m(i,j,k) \in \Sk(\mathcal{H}_2)\) for any $\lp \in \{A_1, A_2, A_3, B_{1 2}, B_{2 3}, B_{1 3} \}$ with the correspondence given by  
    \[
        \lp \mapsto - \hat{\mathcal{O}}_{\lp} \text{ and } n(i, j, k) \mapsto \alpha^{-1}(i,j,k) \Psi_{i, j, k}.
    \]
    Hence, the Arthamonov--Shakirov, genus $2$, spherical DAHA is isomorphic to the image of the skein algebra $\Sk(\Sigma_{2,0})$ in the endomorphism ring of $\Sk(\mathcal H_2)$.
\end{thm}
\begin{proof}
Under the correspondence 
\[
\hat{\mathcal{O}}_{A_{1}} \cdot \Psi_{i, j, k} = (q^{i/2} t^{1/2} + q^{-i/2} t^{-1/2}) \Psi_{i, j, k} =  (s^{2i+2} + s^{-2i-2}) \Psi_{i, j, k}
\] 
becomes 
\[
-\alpha(i,j,k) A_1 \cdot n(i,j,k)  =  (s^{2i+2} + s^{-2i-2}) \alpha(i,j,k) n(i, j, k).
\] 
Dividing by \(-\alpha(i,j,k)\) gives \(A_1 \cdot n(i,j,k)  = (-s^{2i+2} -s^{-2i-2}) n(i, j, k)\) which from \cref{naction} is indeed \(A_1 \cdot n(i,j,k)\). 
The results for \(\hat{\mathcal{O}}_{A_{2}} \Psi_{i, j, k}\) and \(\hat{\mathcal{O}}_{A_{3}} \Psi_{i, j, k}\) follows by symmetry.

As the triple \((i,j,k)\) is admissible, by \cref{lem:adform} we can assume the triple is of the form \((i+d,j+d,i+j)\). Under the correspondence
\[\hat{\mathcal{O}}_{B_{1 2}} \Psi_{i+d,j+d,i+j} = \sum_{a,b \in \{-1,1\}} C_{a,b}(i+d,j+d,i+j) \Psi_{i+d+a, j+d+b, i+j}\]
maps to
\begin{align*}
 &B_{1 2} \cdot \alpha(i+d,j+d,i+j) n(i+d,j+d,i+j)  \\
 &\;=  -\sum_{a,b \in \{-1,1\}} C_{a,b}(i+d,j+d,i+j) \alpha(i+d+a,j+d+b,i+j) n(i+d+a,j+d+b,i+j) \\
    \iff &B_{1 2} \cdot n(i+d,j+d,i+j) \\
    &\;=  \sum_{a,b \in \{-1,1\}} \frac{-C_{a,b}(i+d,j+d,i+j) \alpha(i+d+a,j+d+b,i+j)}{\alpha(i+d,j+d,i+j)} n(i+d+a,j+d+b,i+j)
\end{align*}
So it suffices to show that coefficient is \(D_{1,1}(i+d,j+d,i+j)\).
When \(a=1\) and \(b=1\):
\begin{align*}
    &- \frac{C_{1,1}(i+d,j+d,i+j) \alpha(i+d+1,j+d+1,i+j)}{\alpha(i+d,j+d,i+j)}\\
    &=(-1)^{2i+2j+2d+2}C_{1,1}(i+d,j+d,i+j) \frac{ \frac{\qint{i+1} ... \qint{i+d+2} \qint{j+1} ... \qint{j+d+2}}{\qint{1} ...\qint{d+1} \qint{i+j+2} ... \qint{i+j+d+2}} }{\frac{\qint{i+1} ... \qint{i+d+1} \qint{j+1} ... \qint{j+d+1}}{\qint{1} ...\qint{d} \qint{i+j+2} ... \qint{i+j+d+1}} } \\
    &= C_{1,1}(i+d,j+d,i+j) \frac{\qint{i+d+2}\qint{j+d+2}}{\qint{d+1}\qint{i+j+d+2}} \\
    &= \frac{\qint{i+j+d+2} \qint{d+1}}{ \qint{i+d+2}  \qint{j+d+2}}b
\end{align*}
The other cases when \(a=-1\) and \(b=-1\), when \(a=1\) and \(b=-1\), and when \(a=-1\) and \(b=1\) are similar.
The result for \(\hat{\mathcal{O}}_{B_{1 3}} \Psi_{i, j, k}\) and \(\hat{\mathcal{O}}_{B_{2 3}} \Psi_{i, j, k}\) follows by symmetry.
\end{proof}

Using Le's theorem \cite{Le21} that the action of the skein algebra of a closed surface on the skein module of a handlebody is faithful, we immediately obtain the following.
\begin{cor}\label{cor:asisoskein}
The $t=q=s^4$ specialization of the Arthamonov-Shakirov algebra is isomorphic to the skein algebra $\Sk(\Sigma_{2,0})$.
\end{cor}

Finally, we give a precise statement relating these algebras to quantum knot invariants. The standard embedding of the handlebody $\mathcal H_2$ into $S^3$ induces a map on the corresponding skein modules. Since the skein module of $S^3$ is just the ground ring $R$, we can view this map as an evaluation function:
\[
\mathrm{ev}: \Sk(\mathcal H_2) \to \Sk(S^3) = R 
\]
This map has been computed explicitly in the $n(i,j,k)$ basis in \cite{masbaum1994}, and for concreteness we recall the formula here. Let $(a,b,c)$ be an admissible triple, and let $(i,j,k)$ be the labels on the internal vertices, which can be written explicitly as $i = (b+c-a)/2$, $j = (a+c-b)/2$, and $k = (a+b-c)/2$.
\begin{thm}[{\cite[Thm. 1]{masbaum1994}}]
The evaluation formula is
\begin{equation}\label{eq:eval}
\mathrm{ev}(n(a,b,c)) = (-1)^{i+j+k} \frac{ [i+j+k+1]![i]![j]![k]! }{[i+j]![j+k]![i+k]!}
\end{equation}
where $[\ell]!$ is the $q$-factorial in the variable $s$.
\end{thm}

Now suppose $\alpha$ is a simple closed curve on $\Sigma_{2,0}$. If we embed $\Sigma_{2,0}$ into $S^3$ in the standard way, this induces an evaluation map 
\[
\mathrm{ev_{\Sigma}}: \Sk(\Sigma_{2,0}) \to \Sk(S^3) = R
\]
Under this embedding we can view $\alpha$ as a knot in $S^3$, and by definition, the Jones polynomial $J(\alpha)$ of $\alpha$ is given by $\mathrm{ev}_\Sigma(\alpha)$. We could instead first embed the knot in the solid handlebody, and 
then embed this in $S^3$, which leads to the following (tautological) identity:
\begin{equation}\label{eq:taut}
    J(\alpha) = \mathrm{ev}_{\Sigma}(\alpha) = \mathrm{ev}(\alpha \cdot n(0,0,0))
\end{equation}

\begin{cor}\label{cor:Jones}
Suppose that $\mathrm{ev}_{q,t}: \mathcal H \to R$ is a linear map whose $t=q$ specialisation is equal to the evaluation map in \eqref{eq:eval}. Suppose that $\alpha_{q,t}$ is an element in the Arthamonov-Shakirov algebra whose $t=q$ specialisation is equal to $\alpha \in \Sk(\Sigma_{2,0})$.
Then the specialisation $\mathrm{ev}_{s^4,s^4}(\alpha_{s^4,s^4}\cdot \Psi(0,0,0))$ is equal to the Jones polynomial of $\alpha$, viewed as a knot in $S^3$ under the standard embedding.
\end{cor}
\begin{proof}
By \cref{cor:asisoskein}, the Arthamonov-Shakirov algebra and its action on $\mathcal H$ specialise at $t=q=s^4$ to the skein algebra of $\Sigma_{2,0}$ and its action on the skein module of the handlebody. This implies 
\[
\mathrm{ev}_{s,s}(\alpha_{s,s}\cdot \Psi(0,0,0)) = \mathrm{ev}(\alpha \cdot n(0,0,0))
\]
The right hand side of this equality is the Jones polynomial of $\alpha$ by equation \eqref{eq:taut}.
\end{proof}

\begin{rmk}
The `correct' evaluation map $\mathrm{ev}_{q,t}$ should depend nontrivially on $q$ and $t$, and should only be equal to the skein-theoretic evaluation after these parameters are set equal. However, this evaluation map isn't defined in \cite{arthamonov&shakirov2017}, and finding the `correct' definition is a nontrivial task. (One could set $\mathrm{ev}_{q,t} := \mathrm{ev}_s$, but this would not be very interesting.) 

We note that one reason a nontrivial $q,t$ deformation of the evaluation map may be interesting comes from knot homology. More precisely, inspired by \cite{AS15}, Cherednik \cite{CerednikJones, CD16} has conjectured that Khovanov homology of (iterated) torus knots can be computed using the $sl_2$ spherical DAHA, its action on the polynomial representation $k[x]$, and the evaluation map given by the standard pairing on the polynomial representation. These are $q,t$ deformations of the skein algebra of the torus, the skein module of the solid torus, and the skein-theoretic evaluation map, respectively (see \cite{Sam19}).
\end{rmk}

\appendix
\appendixpage
\addappheadtotoc
\begin{appendix}
\section{Calculation of Loop Actions}
\label{appendix:cals}
In this appendix we calculate almost all the actions of the loops depicted below on \(\Sk(\mathcal{H}_2)\), the skein module of the genus \(2\) handlebody. 
For the loops in the left figure, we shall use the non-dumbbell basis (or \(\theta\)-basis) for \(\Sk(\mathcal{H}_2)\): \(n(i,j,k)\) for all admissible \((i,j,k)\).
For the loops in the right figure, we use the dumbbell basis for \(\Sk(\mathcal{H}_2)\): \(m(i,j,k)\) for all \(i,j,k\) such that \((i,i,j)\) and \((k, k, j)\) are admissible (see \cref{sec:skeinmodulebasis}).
We leave the calculation of the action of $A_2$ on the basis $m(i, j, k)$ to \cref{appendix:finalloop} as this calculation is significantly more complex than the others. 

\begin{figure}[!h]
\centering
\begin{subfigure}{.4\textwidth}
  \centering
  \diagramhh{figures}{cycles}{3pt}{0pt}{0.3}
\end{subfigure}
\hspace{5mm}
\begin{subfigure}{.4\textwidth}
  \centering
  \diagramhh{figures}{cycles2}{0pt}{0pt}{0.3}
\end{subfigure}
\end{figure}

In \cref{back:wenzl} we defined Jones-Wenzl idempotents and used them to define trivalent vertices. We now outline a number of results which we use as graphical calculus for skein modules, using \cite{masbaum1994} as a reference. 
\begin{lem}{\cite{masbaum1994,kauffmanBook1994}}
\label{lem:diagramprops}
Let $m$ and $n$ be integers.

\begin{tabular}{@{}p{.35\textwidth}@{\hspace{.05\textwidth}}p{.5\textwidth}@{}}
\begin{equation} 
    \diagramhh{properties}{properties1}{16pt}{0pt}{0.75} = \diagramhh{properties}{properties2}{16pt}{0pt}{0.75},  
    \label{idempotent} 
\end{equation}
&
\begin{equation} 
    \diagramhh{properties}{properties5}{16pt}{0pt}{0.75} = 0, 
    \label{turnback} 
\end{equation}
\\
\begin{equation}
    \diagramhh{properties}{properties3}{16pt}{0pt}{0.75} = \diagramhh{properties}{properties4}{16pt}{0pt}{0.75}, 
    \label{property1}
\end{equation}
&
\begin{equation} 
    \diagramhh{properties}{bookloop1}{16pt}{0pt}{0.75} =  (-s^{2m+2}-s^{-2m-2}), \diagramhh{properties}{bookloop2}{16pt}{0pt}{0.75} 
    \label{bookloop}
\end{equation}
\\
\begin{equation}
    \diagramhh{properties}{prop62}{16pt}{0pt}{0.75} = (-1)^m s^{m(m+2)} \diagramhh{properties}{bookloop2}{16pt}{0pt}{0.75}, 
    \label{prop6a} 
\end{equation}
&
\begin{equation} 
    \diagramhh{properties}{prop64}{16pt}{0pt}{0.75} = (-1)^m s^{-m(m+2)}\diagramhh{properties}{bookloop2}{16pt}{0pt}{0.75} 
    \label{prop6b}
\end{equation}
\end{tabular}
\end{lem}

\begin{lem}[\cite{masbaum1994}]
\begin{align} 
    \diagramhh{properties}{paperleaf1}{15pt}{0pt}{0.75} &= \frac{\qint{i}}{\qint{i+j}}\diagramhh{properties}{paperleaf2}{15pt}{0pt}{0.75}
    \label{paperleaf}
    \\
    \diagramhh{properties}{papercircstrand1}{15pt}{0pt}{0.75} &= \frac{\qint{i+j+k+1}\qint{j}}{\qint{i+j}\qint{k+j}} \diagramhh{properties}{papercircstrand2}{15pt}{0pt}{0.75}
    \label{papercircstrand}
\end{align}
\end{lem}

Calculating the actions of \(A_1\), \(A_2\) and \(A_3\) on $n(i, j, k)$ follows directly from \cref{lem:diagramprops} \eqref{bookloop}
\begin{prop}
\label{prop:AactionN}
\begin{align*}
    A_1 \cdot n(i,j,k) &= (-s^{2i+2} -s^{-2i-2}) n(i,j,k) \\
    A_2 \cdot n(i,j,k) &= (-s^{2j+2} -s^{-2j-2}) n(i,j,k) \\
    A_3 \cdot n(i,j,k) &= (-s^{2k+2} -s^{-2k-2}) n(i,j,k) \\
\end{align*}
\end{prop}
\begin{proof} This follows immediately from \cref{lem:diagramprops}.
\end{proof}

The actions of \(A_1\), \(A_3\) and \(X\) on $m(i,j,k)$ also follow directly from \cref{lem:diagramprops} \eqref{bookloop}.
\begin{prop}
\label{prop:AXactionM}
For all \(i,j,k\) such that the triples \((i, i, j)\) and \((k, k, j)\) are admissible, 
\begin{align}
X \cdot m(i, j, k) & = (-s^{2j+2} - s^{-2j-2}) m(i, j,k ), \\
A_1 \cdot m(i, j, k) &=  (-s^{2i+2} - s^{-2i-2}) m(i, j,k ) \\
A_3 \cdot m(i, j, k) &= (-s^{2k+2} - s^{-2k-2}) m(i, j,k )
\end{align}
\end{prop}
\begin{proof}
This follows from \cref{lem:diagramprops} \eqref{bookloop}, for example 
\[X \cdot m(i, j, k) = \smalldiagram{alt}{alt0} \stackrel{\eqref{bookloop}}{=}  (-s^{2j+2}-s^{-2j-2}) \smalldiagram{alt}{alt4}.\]
\end{proof}

We shall now calculate the action of $B_{12}$, $B_{13}$ and $B_{23}$ on the $n(i,j,k)$ basis of $\Sk(\mathcal{H}_2)$.

\begin{defn}
Let $(i, j, k)$ be an admissible triple. We define the coefficients $D_{a,b}(i,j,k)$ for $a,b = \pm 1$ as follows
\begin{alignat*}{5}
D_{1,-1}(i,j,k) &= -\frac{\qint{\frac{j+k-i}{2}}^2}{ \qint{j} \qint{j+1}} 
&&D_{-1,-1}(i,j,k) &&= \frac{\qint{\frac{i+j+k+2}{2}}^2 \qint{\frac{i+j-k}{2}}^2}{ \qint{i} \qint{i+1} \qint{j} \qint{j+1}} \\
D_{-1,1}(i,j,k)\;&= -\frac{\qint{\frac{i+k-j}{2}}^2}{ \qint{i} \qint{i+1}}
&\qquad \quad& \centermathcell{D_{1,1}(i,j,k)} &&= 1
\end{alignat*}
where the coefficient is defined to be $0$ if the denominator is $0$.
\end{defn}

\begin{prop}
\label{prop:BactionN}
Let $(i, j, k)$ be an admissible triple.
\begin{align*}
    B_{12} \cdot n(i,j,k) &= \sum_{a,b \in \{-1,1\}} D_{a,b}(i,j,k) n(i+a, j+b, k)\\
    B_{13} \cdot n(i,j,k) &= \sum_{a,b \in \{-1,1\}} D_{a,b}(i,k,j) n(i+a, j, k+b)\\
    B_{23} \cdot n(i,j,k) &= \sum_{a,b \in \{-1,1\}} D_{a,b}(j,k,i) n(i, j+a, k+b)
\end{align*}
\begin{proof}
\begin{alignat*}{4}
    &\rlap{$B_{1 2} \cdot n(i,j,k)$}\\
     &\centermathcell{\stackrel{\text{defn}}{=}}&&&\diagramhh{yaction}{yactionm2}{25pt}{0pt}{0.75} &&\centermathcell{\stackrel{\text{isotopy}}{=}}& \diagramhh{yaction}{yactionm3}{25pt}{0pt}{0.75}\\
     &\centermathcell{\stackrel{\text{Wenzl}}{=}}&&& \diagramhh{yaction}{yactionm4}{25pt}{0pt}{0.75} &&- \frac{\qint{i}}{\qint{i+1}}& \;\diagramhh{yaction}{yactionm5}{25pt}{0pt}{0.75} \\
     &\quad& - \frac{\qint{j}}{\qint{j+1}}&& \;\diagramhh{yaction}{yactionm6}{18pt}{0pt}{0.75} &&\;+ \frac{\qint{i}\qint{j}}{\qint{i+1}\qint{j+1}}& \;\diagramhh{yaction}{yactionm7}{25pt}{0pt}{0.75} \\
     &\stackrel{\eqref{paperleaf},\eqref{papercircstrand}}{=}&&& \diagramhh{yaction}{yactionm4}{25pt}{0pt}{0.75} &&- \frac{\qint{b}^2}{\qint{i+1}\qint{i}}&\; \diagramhh{yaction}{yactionm8}{25pt}{0pt}{0.75} \\
     &\quad&- \frac{\qint{c}^2}{\qint{j+1}\qint{j}}\;&& \diagramhh{yaction}{yactionm9}{18pt}{0pt}{0.75} &&+ \frac{\qint{a}^2\qint{a+b+c+1}^2}{\qint{i+1}\qint{j+1}\qint{i}\qint{j}}& \;\diagramhh{yaction}{yactionm10}{25pt}{0pt}{0.75} \\
     &\centermathcell{=}&&& \centermathcell{n(i+1,j+1,k)} &&- \frac{\qint{\frac{i+k-j}{2}}^2}{\qint{i}\qint{i+1}}& \centermathcell{n(i-1, j+1, k)} \\
    &\quad&- \frac{\qint{ \frac{j+k-i}{2}}^2}{\qint{j}\qint{j+1}}&& \centermathcell{n(i+1, j-1, k)} &&+  \frac{\qint{\frac{i+j-k}{2}}^2\qint{\frac{i+j+k+2}{2}}^2}{\qint{i}\qint{i+1}\qint{j}\qint{j+1}}& \centermathcell{n(i+1, j+1, k)}
\end{alignat*}
and the other cases are symmetric.
\end{proof}
\end{prop}

Now we calculate the action of $B_{12}$ and $B_{23}$ on the $m(i,j,k)$ basis of $\Sk(\mathcal{H}_2)$.

\begin{prop}
\label{prop:BactionM}
For all \(i,j,k\) such that the triples \((i, i, j)\) and \((k, k, j)\) are admissible,
\begin{align*}
B_{12} \cdot m(i, j, k) &= {
    \begin{cases} m(i+1, j, k) & \text{ for } i = \frac{j}{2} \\
     m(i+1, j, k) +\frac{\qint{2i+j+1}\qint{i-j/2}}{\qint{i}\qint{i+1}} m(i-1,j,k) & \text{ for } i> j/2 
\end{cases}} \label{eq:mact} \\
B_{23} \cdot m(i, j, k) &= {
    \begin{cases} m(i, j, k+1) & \text{ for } k = \frac{j}{2} \\
     m(i, j, k+1) +\frac{\qint{2k+j+1}\qint{k-j/2}}{\qint{k}\qint{k+1}} m(i,j,k-1) & \text{ for } k> j/2 
\end{cases}}
\end{align*}
\begin{proof}
When \(i \neq \frac{j}{2}\) (and so \(a \neq 0\)) we have:
\begin{alignat*}{5}
    &\rlap{$B_{12} \cdot m(i, j, k)$}   \\
    &\centermathcell{\stackrel{\text{defn}}{=}}& \diagramhh{yaction}{yactionn2}{17pt}{0pt}{0.75} &&\centermathcell{\stackrel{\text{isotopy}}{=}} &&\diagramhh{yaction}{yactionn3}{17pt}{0pt}{0.75} \\
    &\centermathcell{\stackrel{\text{Wenzl}}{=}}& \diagramhh{yaction}{yactionn4}{17pt}{0pt}{0.75} &&-\frac{\qint{i}}{\qint{i+1}}\;&& \diagramhh{yaction}{yactionn5}{17pt}{0pt}{0.75} \\
    &\centermathcell{\stackrel{\text{isotopy}}{=}}& \centermathcell{m(i+1, j, k)} &&-\frac{\qint{i}}{\qint{i+1}}\;&& \diagramhh{yaction}{yactionn6}{21pt}{0pt}{0.75} \\
    &\centermathcell{\stackrel{\eqref{papercircstrand}}{=}}& \centermathcell{m(i+1, j, k)} && + \frac{\qint{i}\qint{a+b+c+1}\qint{a}}{\qint{i+1}\qint{a+b}\qint{a+c}} && \diagramhh{yaction}{yactionn7}{21pt}{0pt}{0.75} \\
    &\centermathcell{=}& \centermathcell{m(i+1, j, k)} &&-\frac{\qint{i-\frac{j}{2}}\qint{2i+j+1}}{\qint{i}\qint{i+1}}\; && \centermathcell{m(i-1,j,k)}
\end{alignat*}
When \(i = \frac{j}{2}\) this implies that \(a = 0\), so at the fourth step the term
\[\diagramhh{yaction}{yactionn6}{21pt}{0pt}{0.75} = 0\]
by \cref{property1} and \cref{turnback}. The result for \(B_{2 3}\) is analogous. 
\end{proof}
\end{prop}

\section{Calculation of $A_2$ Action}
\label{appendix:finalloop}

Finally we need to find \(A_2 \cdot m(i, j, k)\). This is possible using a change of basis:
\begin{prop}
For all \(i,j,k\) such that the triples \((i, i, j)\) and \((k, k, j)\) are admissible,
\[x_2 \cdot m(i, j, k) = \sum_{a,b} \left( -s^{2a + 2} - s^{-2a-2} \right) \left\{ \begin{smallmatrix} i & i & a \\ k & k & j \end{smallmatrix} \right\} \left\{ \begin{smallmatrix} i & k & b \\ k & i & a \end{smallmatrix} \right\} m(i, b, k)\]
where the sum is over all \(a,b\) such that \(\{i,k,a\}, \{i, i, b\},\) and \(\{k,k,b\}\) are admissible. 
\begin{proof}
\begin{alignat*}{4}
\smalldiagram{alt}{alt1} \;&=& \sum_a \left\{ \begin{smallmatrix} i & i & a \\ k & k & j \end{smallmatrix} \right\}\;&\smalldiagram{alt}{alt3} & \;(\cref{thm:ChangeofBasis})\\
&=& \sum_a \left( -s^{2a + 2} - s^{-2a-2} \right) \left\{ \begin{smallmatrix} i & i & a \\ k & k & j \end{smallmatrix} \right\}\;&\smalldiagram{alt}{alt2}\; & \eqref{bookloop}\\
&=& \sum_{a,b} \left( -s^{2a + 2} - s^{-2a-2} \right) \left\{ \begin{smallmatrix} i & i & a \\ k & k & j \end{smallmatrix} \right\} \left\{ \begin{smallmatrix} i & k & b \\ k & i & a \end{smallmatrix} \right\} &\;\smalldiagram{alt}{alt4} & \;(\cref{thm:ChangeofBasis})
\end{alignat*}
where the sum is over all \(a,b\) such that \(\{i,k,a\}, \{i, i, b\},\) and \(\{k,k,b\}\) are admissible. 
\end{proof}
\end{prop}
However, this result is not very explicit, it does not even allow one to easily see how many terms there are, and is not sufficient for our purposes, so we shall compute the result directly. (We note that while the approach below works, there are other approaches that may be more efficient; for example, similar computations were done in \cite{MP15} using the fusion rules in the appendix of \textit{loc. cit.}.)
\begin{thm}
\label{thm:appendix}
Let \(i,j,k\) be such that the triples \((i, i, j)\) and \((k, k, j)\) are admissible
The action of the middle $A$-loop $A_2$ on the $m$-basis of $\Sk(\mathcal{H}_2)$ is given by
\begin{alignat*}{3}
    A_2 \cdot m(i,j,k)= &&{}-\bigg((s^2-s^{-2})^2 \frac{\qint{j/2}^4\qint{i+j/2+1}\qint{k + j/2+1}}{\qint{j-1}\qint{j}^2\qint{j+1}}\bigg) &m(i, j-2, k) \\
    &&-\bigg((s^2-s^{-2})^2\qint{i- j/2}\qint{k - j/2}\bigg)&m(i, j+2, k)\\
    &&+\bigg(\big(- s^{-2(i+k+1)} - s^{2(i+k+1)}\big) + K\bigg) &m(i,j,k)
\end{alignat*}
where 
\[K = (s^2-s^{-2})^2\left(\frac{\qint{j/2 +1}^2\qint{i- j/2}\qint{k -j/2}}{\qint{j+1}\qint{j+2}} +  \frac{\qint{j/2}^2\qint{i+j/2+1}\qint{k + j/2+1}}{\qint{j}\qint{j+1}}\right)\]
\end{thm}

\begin{rmk}
In the above expressions, we have used the convention that if $j=0$, then any coefficient term with $[j]$ in the denominator is $0$. In fact, this follows by inspection of these coefficients, since the term $[j/2]^2_s/[j]_s$ (and hence each coefficient itself) is equal to zero when $j=0$.
\end{rmk}
For the proof of \cref{thm:appendix} we first need some lemmas.

\begin{lem}
\label{2twist}
\[\smalldiagram{2twist}{2twist1}= s^{m-1} \smalldiagram{2twist}{2twist6},\quad \smalldiagram{2twist}{2twist2} = s^{-m+1}\smalldiagram{2twist}{2twist6}\]
\begin{proof}
\begin{align*}
\smalldiagram{2twist}{2twist3} &\stackrel{(Skein)}{=} s^{-1} \underbrace{\smalldiagram{2twist}{2twist5}}_{0 \text{ as } (\dagger)} + \underbrace{s \smalldiagram{2twist}{2twist4}}_{\text{crossing has moved one position over}}\\
&\quad= \dots = s^{m-1} \smalldiagram{2twist}{2twist6}
\end{align*}
Where \((\dagger)\) is 
\[\smalldiagram{2twist}{2twist5} \stackrel{(Skein)}{=} s^{-1} \smalldiagram{2twist}{2twist7} + s \underbrace{\smalldiagram{2twist}{2twist8}}_{0 \text{ as turnback}} = \dots = s^{3-m} \underbrace{\smalldiagram{2twist}{2twist9}}_{0 \text{ as turnback}} = 0 .\]
\end{proof}
\end{lem}

\begin{lem}
\label{1stand}
\[\smalldiagram{1strand}{1strand1} =  s^{-m} \smalldiagram{1strand}{1strand9} + s^{m} \smalldiagram{1strand}{1strand8}\]
\[\smalldiagram{1strand}{1strand10} = s^{m} \smalldiagram{1strand}{1strand12} + s^{-m} \smalldiagram{1strand}{1strand11}\]
\begin{proof}
\begin{align*}
    \smalldiagram{1strand}{1strand1} &\stackrel{(Skein)}{=} s^{-1} \smalldiagram{1strand}{1strand3} + s \smalldiagram{1strand}{1strand2} \\
    &\stackrel{(Skein)}{=} s^{-2} \smalldiagram{1strand}{1strand7} + \underbrace{\smalldiagram{1strand}{1strand6}}_{0 \text{ as turnback}} + \underbrace{\smalldiagram{1strand}{1strand5}}_{0 \text{ as turnback}} + s^{2} \smalldiagram{1strand}{1strand4} \\
    &\quad = \dots = s^{-m} \smalldiagram{1strand}{1strand9} + s^{m} \smalldiagram{1strand}{1strand8}
\end{align*}
The second result is analogous.
\end{proof}
\end{lem}

\begin{lem}
\[\diagramhh{untwist}{untwist1}{22pt}{0pt}{0.75} = -s^{3+i}\; \diagramhh{untwist}{untwist3}{26pt}{0pt}{0.75}, \quad \diagramhh{untwist}{untwist5}{22pt}{0pt}{0.75} = s^{-i} \;\diagramhh{untwist}{untwist6}{22pt}{0pt}{0.75}\]
\[\diagramhh{untwist}{untwist7}{22pt}{0pt}{0.75} = -s^{-3-k}\;\diagramhh{untwist}{untwist9}{25pt}{0pt}{0.75},\quad \diagramhh{untwist}{untwist8}{22pt}{0pt}{0.75} = s^{k} \;\diagramhh{untwist}{untwist10}{22pt}{0pt}{0.75}\]
\label{untwist}
\begin{proof}
\begin{align*}
    \diagramhh{untwist}{untwist1}{22pt}{0pt}{0.75} &\stackrel{\eqref{prop6a}}{=} -s^{-3}\; \diagramhh{untwist}{untwist2}{22pt}{0pt}{0.75} \stackrel{\eqref{2twist}}{=} -s^{-3+i} \; \diagramhh{untwist}{untwist3}{26pt}{0pt}{0.75} = -s^{-3+i} \; \diagramhh{untwist}{untwist4}{24pt}{0pt}{0.75} \\
    \diagramhh{untwist}{untwist5}{22pt}{0pt}{0.75} &\stackrel{\eqref{2twist}}{=} s^{-i} \; \diagramhh{untwist}{untwist6}{22pt}{0pt}{0.75}
\end{align*}
The final two cases are analogous.
\end{proof}
\end{lem}

\begin{lem}
\label{outloop}
\[\diagramhh{outloop}{outloop1}{25pt}{0pt}{0.75} = \begin{cases} \smalldiagram{outloop}{outloop8} & \text{when } j=0 \\[20pt]
\smalldiagram{outloop}{outloop8} - \frac{\qint{\frac{j+1}{2}}^2}{\qint{j}\qint{j+1}} \smalldiagram{outloop}{outloop9}  & \text{ for } j>0 \end{cases}\]
\begin{proof}
When \(j\neq0\) we have
\begin{alignat*}{5}
&\rlap{$\diagram{outloop}{outloop1}$} \\
&\centermathcell{\stackrel{\text{isotopy}}{=}}\;& \centermathcell{\mediumdiagram{outloop}{outloop2}} &&\centermathcell{\;=\;}& \centermathcell{\mediumdiagram{outloop}{outloop3}} &\\
&\stackrel{\text{Wenzl}}{=}& \centermathcell{\mediumdiagram{outloop}{outloop4}} &&- \frac{\qint{j}}{\qint{j+1}} &\centermathcell{\mediumdiagram{outloop}{outloop5}} \\
&\centermathcell{\stackrel{\eqref{paperleaf}}{=}}& \centermathcell{\smalldiagram{outloop}{outloop8}}  &&- \frac{\qint{j-b}}{\qint{j}} \frac{\qint{j-e}}{\qint{j}} \frac{\qint{j}}{\qint{j+1}} &\centermathcell{\mediumdiagram{outloop}{outloop7}} \\
&\centermathcell{=} & \centermathcell{\smalldiagram{outloop}{outloop8}}  &&- \frac{\qint{\frac{j-1}{2}}^2}{\qint{j}\qint{j+1}} &\centermathcell{\smalldiagram{outloop}{outloop9}} \\
\end{alignat*}
as \(b=e= \frac{j-1}{2}\). When \(j=0\) we must have \(b = c = 0\). We apply the isotopy as before but instead of applying the Wenzl recurrence relation we just add another box to the red strand and reverse the isotopy to give the result.
\end{proof}
\end{lem}

\begin{lem}
\label{inloop}
When \(j=0\) 
\[\diagramhh{inloop}{inloop1}{25pt}{0pt}{0.75} =  \smalldiagram{inloop}{inloop6} 
- \frac{1}{\qint{2}} \smalldiagram{inloop}{inloop7}\]
and when \(j\neq0\) 
\begin{alignat*}{3}
&\rlap{$\diagram{inloop}{inloop1}$} \\
&=& \frac{\qint{i- \frac{j}{2} +1}\qint{k - \frac{j}{2}+1}}{\qint{i+1}\qint{k+1}} \smalldiagram{inloop}{inloop6} \\
&&\hspace{-2em}- \left(\frac{\qint{\frac{j+2}{2}}^2\qint{i- \frac{j}{2} +1}\qint{k - \frac{j}{2}+1}}{\qint{i+1}\qint{j+1}\qint{j+2}\qint{k+1}} +  \frac{\qint{\frac{j}{2}}^2\qint{i+\frac{j}{2}+2}\qint{k + \frac{j}{2}+2}}{\qint{i+1}\qint{j}\qint{j+1}\qint{k+1}}\right) \smalldiagram{inloop}{inloop7} \\
&&+ \frac{\qint{\frac{j}{2}}^4\qint{i+\frac{j}{2}+2}\qint{k + \frac{j}{2}+2}}{\qint{i+1}\qint{j-1}\qint{j}^2\qint{j+1}\qint{k+1}} \smalldiagram{inloop}{inloop8}
\end{alignat*}

\begin{proof}
When \(j \neq 0\) we have
\begin{alignat*}{3}
&\rlap{$\diagram{inloop}{inloop1}$} \\
&\centermathcell{\stackrel{\text{Wenzl}}{=}} \rlap{\hspace{4em}$\diagramhh{inloop}{inloop2}{24pt}{0pt}{0.75}$} &\llap{$- \frac{\qint{j}}{\qint{j+1}}\diagramhh{inloop}{inloop3}{24pt}{0pt}{0.75}$} \\
&\centermathcell{\stackrel{\eqref{paperleaf}, \eqref{papercircstrand}}{=}}& \frac{\qint{a}}{\qint{a+b}} \frac{\qint{d}}{\qint{d+e}} \diagramhh{inloop}{inloop4}{24pt}{0pt}{0.75} \\
&&- \frac{\qint{j}}{\qint{j+1}} \frac{\qint{a+b+c+1}\qint{b}}{\qint{b+a}\qint{b+c}} \frac{\qint{d+e+f+1}\qint{e}}{\qint{e+d}\qint{e+f}} \diagramhh{inloop}{inloop5}{24pt}{0pt}{0.75} \\
&\centermathcell{=}& \frac{\qint{i- \frac{j}{2} +1}\qint{k - \frac{j}{2}+1}}{\qint{i+1}\qint{k+1}} \diagramhh{inloop}{inloop4}{24pt}{0pt}{0.75} \\
&&-  \frac{\qint{\frac{j}{2}}^2\qint{i+\frac{j}{2}+2}\qint{k + \frac{j}{2}+2}}{\qint{i+1}\qint{j}\qint{j+1}\qint{k+1}} \diagramhh{inloop}{inloop5}{24pt}{0pt}{0.75} \\
&\rlap{ as $a= \frac{2i+2-j}{2}, d=\frac{2k+2-j}{2}, b=c=e=f=\frac{j}{2}$}\\
&\centermathcell{\stackrel{\cref{outloop}}{=}} &\frac{\qint{i- \frac{j}{2} +1}\qint{k - \frac{j}{2}+1}}{\qint{i+1}\qint{k+1}} \smalldiagram{inloop}{inloop6} \\
&&\hspace{-6em}- \left(\frac{\qint{\frac{j+2}{2}}^2\qint{i- \frac{j}{2} +1}\qint{k - \frac{j}{2}+1}}{\qint{i+1}\qint{j+1}\qint{j+2}\qint{k+1}} +  \frac{\qint{\frac{j}{2}}^2\qint{i+\frac{j}{2}+2}\qint{k + \frac{j}{2}+2}}{\qint{i+1}\qint{j}\qint{j+1}\qint{k+1}}\right) \smalldiagram{inloop}{inloop7} \\
&&+  \frac{\qint{\frac{j}{2}}^4\qint{i+\frac{j}{2}+2}\qint{k + \frac{j}{2}+2}}{\qint{i+1}\qint{j-1}\qint{j}^2\qint{j+1}\qint{k+1}} \smalldiagram{inloop}{inloop8}
\end{alignat*}
The case when \(j=0\) is similar except at the first step one does not have to apply the Wenzl relation so there is no second term. 
\end{proof}
\end{lem}

We can now use these lemmas to prove \cref{thm:appendix}
\begin{proof}
\begin{alignat*}{5}
&\rlap{$\smalldiagram{alt}{alt1}$}\\
&\centermathcell{\stackrel{(\cref{1stand})}{=}}& s^{i-k} \; \diagram{main}{main2} &&+ s^{-i-k} &&\centermathcell{\diagram{main}{main3}} \\
&&+ s^{i+k} \;  \diagram{main}{main4} &&+ s^{-i+k}&&\centermathcell{\diagram{main}{main5}} \\
&\centermathcell{\stackrel{(\cref{untwist})}{=}}&  s^{2i-2k} \; \diagram{main}{main6} &&- s^{-2i-2k-1} &&\centermathcell{\diagram{main}{main7}} \\
&&- s^{2i+2k +1} \; \diagram{main}{main8} &&+ s^{-2i+2k} &&\centermathcell{\diagram{main}{main9}} \\
&\centermathcell{\stackrel{(Skein)}{=}}&&& \llap{$\left(- s^{-2i-2k-2}
- s^{2i+2k+2} \right)$}&&\centermathcell{\smalldiagram{main}{main11}} \\
&&&&\llap{$+\left( s^{2i-2k} - s^{-2i-2k}
- s^{2i+2k} + s^{-2i+2k}\right)\;$}&&\centermathcell{\diagramhh{main}{main10}{24pt}{0pt}{0.75}} \\ 
&\centermathcell{=}& \centermathcell{\hspace{-2em}\left(- s^{-2i-2k-2}
- s^{2i+2k+2} \right)\smalldiagram{main}{main11}} &&-(s^2-s^{-2})^2\qint{i}\qint{k}&&\diagramhh{main}{main10}{24pt}{0pt}{0.75}\\
&\centermathcell{\stackrel{(\cref{inloop})}{=}} &&&\llap{$-(s^2-s^{-2})^2\qint{i- \frac{j}{2}}\qint{k - \frac{j}{2}}$}&& \centermathcell{\smalldiagram{main}{main12}} \\
&&&&\llap{$+\left(\left(- s^{-2(i+k+1)} - s^{2(i+k+1)} \right) \right.$}&&\\
&&&&\llap{$\left.+ (s^2-s^{-2})^2\left(\frac{\qint{\frac{j+2}{2}}^2\qint{i- \frac{j}{2}}\qint{k - \frac{j}{2}}}{\qint{j+1}\qint{j+2}} +  \frac{\qint{\frac{j}{2}}^2\qint{i+\frac{j}{2}+1}\qint{k + \frac{j}{2}+1}}{\qint{j}\qint{j+1}}\right) \right)$} && \centermathcell{\smalldiagram{main}{main11}} \\
&&&&\llap{$- (s^2-s^{-2})^2 \frac{\qint{\frac{j}{2}}^4\qint{i+\frac{j}{2}+1}\qint{k + \frac{j}{2}+1}}{\qint{j-1}\qint{j}^2\qint{j+1}}$} && \centermathcell{\smalldiagram{main}{main13}} \\
&\rlap{assuming $j \neq 0$ and when $j=0$ at this stage we instead get: } \\
&&&&\llap{$-(s^2-s^{-2})^2\qint{i}\qint{k}$} &&\centermathcell{\smalldiagram{main}{main12}} \\
&&&&\llap{$+ \left(\left(- s^{-2i-2k-2} - s^{2i+2k+2} \right)  + (s^2-s^{-2})^2\frac{\qint{i}\qint{k}}{\qint{2}}\right)$}&&\centermathcell{\smalldiagram{main}{main11}}
\end{alignat*}
\end{proof}
\end{appendix}


\printbibliography[heading=bibintoc,title={References}]

\end{document}